\documentclass{article}
\usepackage{epsfig,epsfig}
\usepackage{amsmath}
\usepackage{amssymb}
\usepackage{amsthm}
\usepackage{bm}
\usepackage{mathrsfs}

\usepackage{booktabs,arydshln,multirow}
\usepackage{setspace}
\usepackage{xcolor}
\usepackage{stmaryrd}
\usepackage{cite}

\newtheorem{theorem}{Theorem}[section]

\newtheorem{lemma}{Lemma}[section]

\newtheorem{remark}{Remark}[section]

\numberwithin{equation}{section} 
\newcommand{\normmm}[1]{{\left\vert\kern-0.25ex\left\vert\kern-0.25ex\left\vert #1
    \right\vert\kern-0.25ex\right\vert\kern-0.25ex\right\vert}}

\def\ds{{\rm d}s}
\def\div{{\rm div}}

\def\d{{\rm d}}
\begin{document}
\title{Finite Volume Element Methods for Two-Dimensional Time Fractional Reaction-Diffusion Equations on Triangular Grids}
\date{ }
\author{Zhichao Fang$^{a,\dag}$, Jie Zhao$^{a,b}$, Hong Li$^{a}$, Yang Liu$^{a}$
\\ \small{\emph{$^a$ School of Mathematical Sciences, Inner Mongolia University, Hohhot 010021, China;}}
\\ \small{\emph{$^b$ School of Statistics and Mathematics, Inner Mongolia University of Finance and Economics,}}
\\ \small{\emph{Hohhot, 010070, China}}
}
\date{}
\maketitle
\footnote{\dag Corresponding author.}
\footnote{Email: zcfang@imu.edu.cn, mazcfang@aliyun.com}
\footnote{Original manuscript submitted to the Journal on October 15, 2019}

{\color{black}\noindent\rule[0.5\baselineskip]{\textwidth}{0.5pt} }
\noindent \textbf{Abstract:}
In this paper, the time fractional reaction-diffusion equations with the Caputo fractional derivative are solved
by using the classical $L1$-formula and the finite volume element (FVE) methods on triangular grids.
The existence and uniqueness for the fully discrete FVE scheme are given.
The stability result and optimal \textit{a priori} error estimate in $L^2(\Omega)$-norm are derived,
but it is difficult to obtain the corresponding results in $H^1(\Omega)$-norm,
so another analysis technique is introduced and used to achieve our goal.
Finally, two numerical examples in different spatial dimensions are given to verify the feasibility and effectiveness.
\\
 {\color{black}\noindent\rule[0.5\baselineskip]{\textwidth}{0.5pt} }
\def\REF#1{\par\hangindent\parindent\indent\llap{#1\enspace}\ignorespaces}
\newcommand{\h}{\hspace{1.cm}}
\newcommand{\hh}{\hspace{2.cm}}
\normalsize \vskip 0.2in

\section{Introduction}

In this article, we consider the following time fractional reaction-diffusion equations with the Caputo fractional derivative
\begin{align}\label{1.1}
   \left\{
   \begin{array}{ll}
        {}_0^CD_t^{\alpha}u(\bm{x},t)-\div(\mathcal{A}(\bm{x})\nabla u(\bm{x},t))+q(\bm{x})u(\bm{x},t)
        =f(\bm{x},t),& (\bm{x},t)\in\Omega\times J,\\
        u(\bm{x},t)=0, &  (\bm{x},t)\in\partial\Omega\times \bar{J}, \\
        u(\bm{x},0)=u_0(\bm{x}),& \bm{x}\in\bar{\Omega},
   \end{array}\right.
\end{align}
where $\Omega\subset R^2$ is a bounded convex polygonal domain,
$\partial\Omega$ is the corresponding boundary, $J=(0,T]$ is the time interval with $0<T<\infty$.
The source function $f(\bm{x},t)$, reaction term coefficient $q(\bm{x})$ and initial data $u_0(\bm{x})$
are smooth enough, and $q(\bm{x})\geq 0$, $\forall \bm{x}\in \bar{\Omega}$.
Moreover, we assume that the diffusion coefficient $\mathcal{A}(\bm{x})=\{a_{i,j}(\bm{x})\}_{2\times2}$ is a sufficiently smooth matrix function, which satisfies symmetric and uniformly positive definite, that is, there exists a
constant $\beta_0>0$ such that
\begin{align*}
   \bm{\xi}^T\mathcal{A}(\bm{x})\bm{\xi}\geq \beta_0\bm{\xi}^T\bm{\xi},\ \forall \bm{\xi}\in R^2,\ \forall \bm{x}\in \bar{\Omega}.
\end{align*}
In \eqref{1.1}, ${}_0^CD_t^{\alpha}(u(\bm{x},t)$ is the Caputo fractional derivative defined by
\begin{align}\label{1.2}
        {}_0^CD_t^{\alpha}u(\bm{x},t)=\dfrac{1}{\Gamma (1-\alpha)}
        \int _{0}^{t} \dfrac{\partial u(\bm{x},s)}{\partial s}\dfrac{\ds}{(t-s)^\alpha}
        ,\ 0<\alpha <1.
\end{align}

Fractional partial differential equations (FPDEs) can describe
various natural phenomena in physics, chemistry and biology \cite{a1,a2,a4} and so on.
Especially when describing materials with some properties such as memory, heterogeneity or heredity, they often have very good results.
Unfortunately, it is difficult to get the exact solutions for most of the FPDEs via the analytical methods.
Thus, in past decades, a lot of numerical methods have been proposed to solve various FPDEs.
\par
For the time FPDEs with the Caputo fractional derivative,
the $L1$-formula which was firstly developed in \cite{Lin-Xu,Sun-Wu} is very useful and popular by many scholars.
Li et al. \cite{Ljch} developed finite element methods with $L1$-formula for solving fractional Maxwell's models.
Li et al. \cite{Licp1} designed a numerical approximation based on $L1$-formula for nonlinear fractional subdiffusion problem, and also discussed a superdiffusion model.
In \cite{Liufw1}, Feng et al. studied the finite element method with a unstructured mesh for the 2D fractional diffusion model with a time-space Riesz fractional derivative.
In \cite{Liufw2}, Feng et al. proposed the finite element method for a novel 2D mixed sub-diffusion and diffusion-wave equation with multi-term time-fractional derivative, in which the fractional derivative was approximated by $L1$-formula and other numerical formulas.
Jiang and Ma \cite{b1} constructed a fully discrete FE scheme based on the $L1$-formula to solve a class of FPDEs,
and gave the error analysis in $L^2$-norm.
Liu et al. \cite{b2} considered an $L1$-formula combined with two-grid mixed element algorithm for solving a nonlinear
fourth-order time fractional reaction-diffusion model.
Liu et al. \cite{b21} developed a finite difference/finite element method with $L1$-formula to treat a nonlinear
time fractional reaction-diffusion equation with a fourth-order derivative term.
Jin et al. \cite{b3} revisited the error analysis of the $L1$-formula,
especially for the nonsmooth initial data, the authors obtained $O(\tau)$ convergence rate.
Li et al. \cite{Li-Huang-Jiang} applied finite element methods based on $L1$-formula to study multi-term fractional diffusion equations,
obtained the unconditional stability and optimal error estimates,
and gave some numerical examples with higher spatial dimension.
Zhao et al. \cite{b4} constructed two conforming and nonconforming MFE schemes with $L1$-formula to solve time fractional diffusion equations,
and gave the stability and convergence results.
Li et al. \cite{b5} provided an FE method based on $L1$-formula to solve time fractional nonlinear parabolic equations,
and obtained optimal error estimates for several fully discrete linearized FE methods for nonlinear equations.
From the current literatures,
we find that there is no report about the finite volume element (FVE) method based on $L1$-formula for solving the FPDEs.
\par
The FVE methods \cite{Ewing-Lazarov-Lin,c1,c2,c3,c4,c5}, also known as box methods \cite{Bank} or generalized difference methods \cite{Li-Li,Li-Chen-Wu}, have been widely used in several fields of science and engineering.
This methods can preserve the local conservation laws for some physical quantities, which is very important in scientific computing.
and have attracted more and more scholars.
Recently, some scholars have done valuable works in solving FPDEs by using FVE methods.
Sayevand and Arjang \cite{d1} proposed a spatially semi-discrete FVE scheme
to solve the time fractional sub-diffusion problem $\frac{\partial^\alpha u}{\partial t^\alpha}=\beta\Delta u+f$ with the Caputo fractional derivative on a rectangular partition, gave some error estimates for the FE and FVE schemes.
Karaa et al. \cite{d2} constructed an FVE scheme for the fractional sub-diffusion equation
$u'+\mathscr{B}^\alpha\mathscr{L}u=f$ in a two-dimensional domain,
where $\mathscr{L}u=-\Delta u$ and $\mathscr{B}^\alpha$ is the Riemann-Liouville fractional derivative in time,
and the authors applied a piecewise linear discontinuous Galerkin method in time and an FVE method in space to construct a fully discrete scheme,
and gave the convergence analysis and numerical experiments.
Karaa and Pani \cite{d3} proposed an FVE scheme for fractional order evolution equations $u'+\partial_t^{1-\alpha}Au=0$,
where $Au=-\Delta u$ and $\partial_t^{1-\alpha}$ is the Riemann-Liouville fractional derivative.
In \cite{d3}, the authors gave the error analysis for the semi-discrete scheme with smooth, middly smooth and nonsmooth data,
and constructed and analyzed two fully discrete schemes by introducing convolution quadrature in time for smooth and nonsmooth initial data,
which were generated by the backward Euler and the second-order difference methods.
\par
In this article, we will construct a fully discrete FVE scheme
to treat the time fractional reaction-diffusion equation \eqref{1.1} on triangular grids
by using the $L1$-formula.
In spatial discretization, we first construct the primal and dual partitions,
select the \textit{trial} function space (based on primal partitions) and the \textit{test} function space (based on dual partitions),
then integrate the original equation \eqref{1.1} under the control volumes
to construct the FVE scheme by using the projection operator $I_h^*$.
In our theoretical analysis,
we give the existence and uniqueness for the fully discrete solution,
derive the stability results in $L^2(\Omega)$-norm and $H^1(\Omega)$-norm,
and obtain the optimal \textit{a priori} error estimates for $u$ in $L^2(\Omega)$-norm and $H^1(\Omega)$-norm.
Moreover, we provide two numerical examples in different spatial dimensions, and give some numerical results
to examine the feasibility and effectiveness of the fully discrete FVE scheme.
Here, because the bilinear $a(\cdot,I_h^*\cdot)$ does not necessarily satisfy symmetry,
it is difficult to obtain the stability and the optimal error result in $H^1(\Omega)$-norm,
so we give another analysis technique to achieve our goal.
\par
The layout of this paper is as follows.
By introducing the operator $I_h^*$ and the $L1$-formula of approximating the Caputo fractional derivative.
a fully discrete FVE scheme for the time fractional reaction-diffusion equation is proposed in Section 2.
In Section 3, we give the truncation errors of $L1$-formula and some properties of the operator $I_h^*$,
and give some important lemma which will be used in theoretical analysis.
In Sections 4 and 5, we give the theoretical analysis for the FVE scheme in detailed,
including the existence, uniqueness, stability and error estimates.
In Section 6, we give two numerical examples in different spatial dimensional to verify the feasibility and effectiveness.

\section{Fully Discrete FVE Scheme}

We use some general definitions and notations of the Sobolev spaces as in Reference \cite{Adams}.
Let $W^{m,p}(\Omega)$ ($m\geq 0$ and $1\leq p\leq \infty$) be the usual Sobolev space defined in $\Omega$
with the norm $\|\cdot\|_{W^{m,p}}$ (abbreviated as $\|\cdot\|_{m,p}$).
When $p=2$, we denote $W^{m,2}(\Omega)$ by $H^m(\Omega)$ with the norm $\|\cdot\|_{H^m(\Omega)}$ (abbreviated as $\|\cdot\|_m$), and denote $H^0(\Omega)$ by $L^2(\Omega)$ with the inner product $(\cdot,\cdot)$ and the norm $\|\cdot\|_{L^2(\Omega)}$ (abbreviated as $\|\cdot\|$).
We also use $H_0^1(\Omega)=\{w\in H^1(\Omega) : w|_{\partial \Omega}=0\}$.
Furthermore, throughout the article, we use the mark $C$ to denote a generic positive constant, which is independent of
spatial and temporal mesh.

The variational formulation of the problem \eqref{1.1} is to find $u(t)\in H_0^1(\Omega)$ such that
\begin{align}\label{2.0}
        ({}_0^CD_t^{\alpha} u,v)+a(u,v)+(qu,v)=(f,v),\ v\in H_0^1(\Omega),
\end{align}
where $a(u,v)=\int_{\Omega}\mathcal{A}\nabla u\cdot\nabla v\d \bm{x},\ \forall u,v\in H_0^1(\Omega)$.

Now, let $\mathcal{T}_h=\{K\}$ be a set of quasi-uniform triangular mesh of the domain $\Omega$ with $h=\max\{h_{K}\}$,
referring to Figure \ref{fig1},
where $h_{K}$ denote the diameter of the triangle $K\in \mathcal{T}_h$.
Then $\overline{\Omega}=\cup_{K\in \mathcal{T}_h}K$ and $Z_h$ denotes all vertices, that is
\begin{align*}
   Z_h=\{z:z \ \text{is a vertex of the element}\ K,  K\in \mathcal{T}_h\}.
\end{align*}
And $Z_h^0\subset Z_h$ denotes the set of all interior vertices in $\mathcal{T}_h$.

\begin{figure}[h]
   \begin{center}
      \includegraphics[width=7cm]{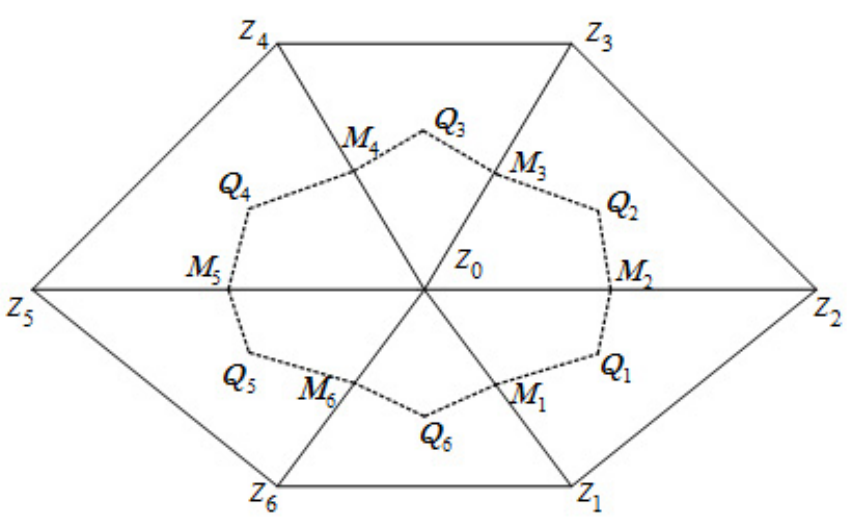}
   \end{center}
   \caption{Primal and dual partitions}\label{fig1}
\end{figure}

Next, let $\mathcal{T}_h^*$ be the dual mesh based on the primary mesh $\mathcal{T}_h$.
With $z_0\in Z_h^0$ as an interior node, let $z_i$ $(i=1,2,\ldots,m)$ be corresponding
adjacent nodes (as shown in Figure \ref{fig1}, $m=6$).
we denote the midpoints of $\overline{z_0z_i}$ by $M_i$,
and denote the barycenters of the triangle $\triangle z_0z_iz_{i+1}$ by $Q_i$, where $(i=1,2,\ldots,m)$
and $z_{m+1}=z_1$.
We construct the \textit{control volume} $K_{z_0}^*$ joining successively $M_1,Q_1,\ldots,M_m,Q_m,M_1$.
Then, the dual mesh $\mathcal{T}_h^*$ is defined by the union of the control volumes $K_{z_i}^*$.
We denote $Q_i,i=1,2,\ldots,m$ be all nodes of the control volume $K_{z_i}^*$,
and denote $Z_h^*=\{Q : Q\ \text{is a node of a control volume}\  K_z^*, K_z^*\in \mathcal{T}_h^*\}$.

Then, we define the following finite element space $V_h$ as the \textit{trial} function space
\begin{align*}
   V_h=\{v\in H_0^1(\Omega) : v|_K\in P_1(K),\ \forall K\in \mathcal{T}_h\},
\end{align*}
and define $V_h^*$ as the \textit{test} function space, that is
\begin{align*}
   V_h^*=\{v\in L^2(\Omega):v|_{K_{z}^*}\in P_0(K_z^*),\ \forall K_{z}^*\in \mathcal{T}_h^*, \
   \text{and}\ v|_{\partial\Omega}=0\}.
\end{align*}
It is obvious that $V_h=\text{span}\{\Phi_z(\bm{x}) : z\in Z_h^0\}$
and $V_h^*=\text{span}\{\Psi_z(\bm{x}) : z\in Z_h^*\}$,
where $\Phi_z$ is the standard linear basis function with the node $z$,
and $\Psi_z$ is the characteristic function of the control volume $K_z^*$.
Next, we will introduce two important interpolation operators.
Let $I_h : C(\Omega)\rightarrow V_h$ be the classical piecewise linear interpolation operator
and $I_h^* : C(\Omega)\rightarrow V_h^*$ be the piecewise constant interpolation operator,
which are defined by
\begin{align*}
      I_hv(\bm{x})=\sum_{z\in Z_h^0}v(z)\Phi_z(\bm{x}),\ \ \text{and}\ \ I_h^*v(\bm{x})=\sum_{z\in Z_h^0}v(z)\Psi_z(\bm{x}).
\end{align*}
From the Reference \cite{Li-Chen-Wu}, we can see that $I_h$ and $I_h^*$ have the following approximation property
\begin{align}
    &\|v-I_hv\|_j\leq Ch^{2-j}\|v\|_2,\ j=0,1,\ \forall v \in H^2(\Omega),\label{2.1}\\
    &\|v-I_h^*v\|\leq Ch\|v\|_1,\ \forall v \in H^1(\Omega).\label{2.2}
\end{align}

Now, integrating \eqref{1.1} on a control volume $K_z^*$ associated with a vertex $z\in Z_h$, and
applying the Green formula, we can get
\begin{align}\label{2.3}
     \int_{K_z^*} {}_0^CD_t^{\alpha} u\d \bm{x}-
      \int_{\partial K_z^*}\mathcal{A}\nabla u\cdot\bm n\d s
      +\int_{K_z^*}qu\d \bm{x}
      =\int_{K_z^*}f\d \bm{x},
\end{align}
where $\bm n$ means the outer-normal direction on $\partial K_z^*$.

Make use of the operator $I_h^*$ to rewrite \eqref{2.3} as the following formulation
\begin{align}\label{2.4}
      ({}_0^CD_t^{\alpha} u,I_h^*v_h)+a(u,I_h^*v_h)
      +(qu,I_h^*v_h)=(f,I_h^*v_h),\ \forall v_h\in V_h,
\end{align}
where $a(\cdot,\cdot)$ defined in \eqref{2.0}, following the References \cite{Li-Chen-Wu,Ewing-Lazarov-Lin}, can be rewritten as follows
\begin{align}\label{2.5}
      a(\bar{u},\bar{v})=
      \left\{
      \begin{array}{ll}
        -\sum_{z\in Z_h} \bar{v}(z)\int_{\partial K_z^*}\mathcal{A}\nabla \bar{u}\cdot\bm n\d s, & \forall \bar{u}\in V_h, \bar{v}\in V_h^*,\\
        \int_{\Omega} \mathcal{A}\nabla \bar{u}\cdot\nabla \bar{v}\d \bm{x}, & \forall \bar{u},\bar{v}\in H_0^1(\Omega).
   \end{array}\right.
\end{align}

Next, we make use of the $L1$-formula to approximate the Caputo fracional derivative.
First, we give a equidistant partition of the time interval $\overline{J}=[0,T]$ by
$0=t_0<t_1<\cdots<t_N=T$, where $t_n=n\tau$, $n=0,1,\cdots ,N$, and $\tau=T/N$ for some positive integer $N$.
For a given function $\varphi$ on $[0,T]$,
let $\varphi^n =\varphi(t_n)$ and $\partial_t\varphi^n=\frac{\varphi^n-\varphi^{n-1}}{\tau}$.
Following References \cite{Lin-Xu,Sun-Wu}, we can approximate the Caputo time fractional derivative
${}_0^CD_t^{\alpha} u(\bm{x},t)$ at $t=t_n$ as follows
\begin{equation}\label{2.6}
    \begin{split}
     {}_0^CD_t^{\alpha} u(\bm{x},t_n)
     &=\frac{1}{\Gamma(1-\alpha)}\int_0^{t_n}\frac{\partial u(\bm{x},s)}{\partial s}\frac{\d s}{(t_n-s)^\alpha}\\
     &=\frac{\tau^{1-\alpha}}{\Gamma(2-\alpha)}\sum_{k=0}^{n-1}b_k\frac{u(\bm{x},t_{n-k})-u(\bm{x},t_{n-k-1})}{\tau}
     +R_1^n(\bm{x})+R_2^n(\bm{x})\\
     &=\frac{\tau^{1-\alpha}}{\Gamma(2-\alpha)}\sum_{k=0}^{n-1} b_k\partial _t u^{n-k}
     +R_t^n(\bm{x})\\
     &=\frac{\tau^{-\alpha}}{\Gamma(2-\alpha)}\sum_{k=0}^{n}\tilde{b}_k^n u^k
     +R_t^n(\bm{x}),
   \end{split}
\end{equation}
where $b_k=(k+1)^{1-\alpha}-k^{1-\alpha}$, $\tilde{b}_n^n=1$, $\tilde{b}_0^n=(n-1)^{1-\alpha}-n^{1-\alpha}$,
$\tilde{b}_k^n=b_{n-k}-b_{n-k-1}$ $(0<k<n)$,
 $R_t^n(\bm{x})=R_1^n(\bm{x})+R_2^n(\bm{x})$, and
\begin{equation}\label{2.7}
    \begin{split}
     R_1^n(\bm{x})
     &=\frac{1}{\Gamma(1-\alpha)}\int_0^{t_n}\frac{\partial u(\bm{x},s)}{\partial s}\frac{\d s}{(t_n-s)^\alpha}
     -\frac{\tau^{1-\alpha}}{\Gamma(2-\alpha)}\sum_{k=0}^{n-1}
     b_k\frac{\partial}{\partial t} u(\bm{x},t_{n-k-1/2}),\\
     R_2^n(\bm{x})
     &=\frac{\tau^{1-\alpha}}{\Gamma(2-\alpha)}\sum_{k=0}^{n-1}
     b_k(\frac{\partial}{\partial t} u(\bm{x},t_{n-k-1/2})-\partial _t u^{n-k}).
   \end{split}
\end{equation}
Denote $D_{\tau}^{\alpha}\varphi^n=\frac{\tau^{-\alpha}}{\Gamma(2-\alpha)}\sum_{k=0}^n\tilde{b}_k^n\varphi^k$,
then we have $ {}_0^CD_t^{\alpha} u(\bm{x},t_n)=D_{\tau}^{\alpha} u^n+R_t^n(\bm{x})$.

Let $u_h^n$ be the fully discrete approximate solution of $u$ at $t=t_n$.
We give the fully discrete FVE scheme
to seek $u_h^n\in V_h,\ (n=0,1,\ldots,N)$, such that
\begin{align}\label{2.8}
      (D_{\tau}^{\alpha} u_h^n,I_h^*v_h)+a(u_h^n,I_h^*v_h)+(qu_h^n,I_h^*v_h)=(f^n,I_h^*v_h),
      \ \forall v_h\in V_h.
\end{align}

\begin{remark}
   Making use of the definition of $D_{\tau}^{\alpha}$, we can rewrite the fully discrete FVE scheme \eqref{2.8}
   as the following or other practical calculation formulation
   \begin{equation}\label{2.9}
      \begin{split}
         &\frac{1}{\Gamma(2-\alpha)}(u_h^n,I_h^*v_h)+\tau^{\alpha} a(u_h^n,I_h^*v_h)+\tau^{\alpha}(qu_h^n,I_h^*v_h)\\
      &\quad=\tau^{\alpha}(f^n,I_h^*v_h)-\frac{1}{\Gamma(2-\alpha)}\sum_{k=0}^{n-1}\tilde{b}_k^n(u_h^k,v_h),
      \ \forall v_h\in V_h.
      \end{split}
   \end{equation}
In Section \ref{sec4}, we will use \eqref{2.9} to prove the existence and uniqueness of the discrete solutions.
\end{remark}

\section{Some Lemmas}

First, we give some properties for the bilinear forms $(\cdot,I_h^*\cdot)$ and $a(\cdot,I_h^*\cdot)$,
which are very important in the later theoretical analysis.

\begin{lemma}$^{\text{\cite{Li-Chen-Wu}}}$\label{lem3.1}
The bilinear form $(\cdot,I_h^*\cdot)$ satisfies the following properties
\begin{align}\label{3.1}
   (v_h,I_h^*w_h)=(w_h,I_h^*v_h),\ \forall v_h,w_h\in V_h,
\end{align}
and there exist constants $\mu_1>0$ and $\mu_2>0$ independent of $h$ such that
\begin{align}\label{3.2}
   &(v_h,I_h^*v_h)\geq\mu_1\|v_h\|^2,\ \forall v_h \in V_h,\\
   &(v_h,I_h^*w_h)\leq \mu_2 \|v_h\| \|w_h\|,\ \forall v_h, w_h\in V_h.
\end{align}
\end{lemma}

\begin{lemma}$^{\text{\cite{Li-Chen-Wu,Ewing-Lazarov-Lin}}}$\label{lem3.2}
The bilinear form $a(\cdot,I_h^*\cdot)$ can be expressed as
\begin{align}
   a(v_h,I_h^*w_h)=a_h(v_h,I_h^*w_h)+b_h(v_h,I_h^*w_h),\ \forall v_h,w_h\in V_h,
\end{align}
where $a_h(v_h,I_h^*w_h)=\sum_{K_Q\in \mathcal{T}_h}\{\mathcal{A}(Q)\nabla v_h(Q)\cdot \nabla w_h(Q)\}S_Q$, ($S_Q$ is the area of $K_Q$).
Then, the bilinear form $a_h(\cdot,I_h^*\cdot)$ has the following properties
\begin{align}
   a_h(v_h,I_h^*w_h)=a_h(w_h,I_h^*v_h),\ \forall v_h,w_h\in V_h,
\end{align}
and there exists constants $\mu_3>0$ and $\mu_4>0$ independent of $h$ such that
\begin{align}
   &a_h(v_h,I_h^*v_h)\geq\mu_3\|v_h\|_1^2,\ \forall w_h\in V_h,\\
   &a_h(v_h,I_h^*w_h)\leq \mu_4\|v_h\|_1 \|w_h\|_1,\ \forall v_h, w_h\in V_h.
\end{align}
We also have that there exist a positive constant $\mu_5$ such that
\begin{align}
   &|b_h(v_h,I_h^*w_h)|\leq \mu_5h\|v_h\|_1\|w_h\|_1,\ \forall v_h,w_h\in V_h,\\
   &|a(v_h,I_h^*w_h)-a(w_h,I_h^*v_h)|\leq \mu_5 h \|v_h\|_1\|w_h\|_1,\ \forall v_h, w_h\in V_h.
\end{align}
\end{lemma}
%%%%%%%%%%%%%%%%%%%%%%%%%%%%%%%%%%%%%%%%%%%%%%%%%%%%%%%%%%%%%%%%%%%%%%%%%%%%%%%%%%%%%%%%%%%%%%%%5

\begin{lemma}$^{\text{\cite{Li-Chen-Wu,Ewing-Lazarov-Lin}}}$\label{lem3.3}
   There exist positive constants $h_0, \mu_6$ and $\mu_7$ such that, for $0<h\leq h_0$
\begin{align}
   &a(v_h,I_h^*v_h)\geq\mu_6\|v_h\|_1^2,\ \forall v_h\in V_h,\\
   &a(v_h,I_h^*w_h)\leq \mu_7\|v_h\|_1 \|w_h\|_1,\ \forall v_h, w_h\in V_h.
\end{align}
\end{lemma}

%%%%%%%%%%%%%%%%%%%%%%%%%%%%%%%%%%%%%%%%%%%%%%%%%%%%%%%%%%%%%%%%%%%%%%%%%%%%%%%%%%%%%%%%%%%%%%%5

Next, following the Reference \cite{Lin-Xu},
we give the estimates of the truncation errors $R_1^n$, $R_2^n$ and $R_t^n$ defined by \eqref{2.7},
and two important lemmas for stability and error analysis.

\begin{lemma}$^{\text{\cite{Lin-Xu}}}$\label{lem3.4}
For the truncation error $R_1^n(\bm{x})$, $R_2^n(\bm{x})$ and $R_t^n(\bm{x})$,
there exists a constant $C>0$ independent of $h$ and $\tau$ such that
\begin{equation*}
   \begin{split}
     &\|R_1^n\|\leq C\tau^{2-\alpha},\ \|R_2^n\|\leq C\tau^2,\\
     &\|R_t^n\|\leq C(\tau^2+\tau^{2-\alpha}).
   \end{split}
\end{equation*}
\end{lemma}

\begin{lemma}$^{\text{\cite{Lin-Xu}}}$\label{lem3.5}
Let $\varphi^k\geq0$, $k=0,1,\ldots,N$, $\zeta>0$ be a constant, which satisfy
\begin{align*}
   \varphi^n\leq -\sum_{k=1}^{n-1}\tilde{b}_k^n\varphi^k+\zeta,
\end{align*}
then there exists a constant $C>0$ independent of $\tau$ such that
\begin{align*}
   \varphi^n\leq C\tau^{-\alpha}\zeta,\ n=1,2,\ldots,N.
\end{align*}
\end{lemma}

\begin{lemma}\label{lem3.6}
Let $\varphi^k\geq0$, $k=0,1,\ldots,N$, $\zeta>0$ and $C_0\geq 1$ be two constants, which satisfy
\begin{align}\label{b1}
   \varphi^n\leq -C_0\sum_{k=0}^{n-1}\tilde{b}_k^n\varphi^k+\zeta,
\end{align}
Then, we have
\begin{align}\label{b3}
   \varphi^n\leq C_0^n(\varphi^0+b_{n-1}^{-1}\zeta),\ n=1,2,\cdots,N.
\end{align}
Furthermore, there exists a constant $C>0$ independent of $\tau$ such that
\begin{align}\label{b2}
   \varphi^n\leq C C_0^n(\varphi^0+\tau^{-\alpha}\zeta),\ n=1,2,\ldots,N.
\end{align}
\end{lemma}
\begin{proof}
First, we use the mathematical induction to prove the following result \eqref{b3}.
Choosing $n=1$ in \eqref{b1}, we have
\begin{align}\label{b4}
   \varphi^1\leq C_0(1-b_1)\varphi^0+\zeta\leq C_0(\varphi^0+b_0^{-1}\zeta).
\end{align}
Therefore, \eqref{b3} is proved for the case $n=1$.

Suppose \eqref{b3} holds for all $n=1,2,\cdots,m$. Next, we will prove that \eqref{b3} also holds for the case $n=m+1$.
Choosing $n=m+1$ in \eqref{b1}, we have
\begin{equation}\label{b5}
    \begin{split}
      \varphi^{m+1}
      \leq& -\sum_{k=0}^{m}C_0\tilde{b}_k^{m+1}\varphi^k+\zeta\\
      \leq & \sum_{k=0}^{m}C_0(b_{m-k}-b_{m-k+1})\varphi^k+\zeta\\
      = & \sum_{k=0}^{m}C_0(b_{k}-b_{k+1})\varphi^{m-k}+\zeta\\
      \leq & C_0(1-b_1)\varphi^m
      +\sum_{k=1}^{m-1}C_0(b_{k}-b_{k+1})\varphi^{m-k}
      +C_0b_m\varphi^0+\zeta.
    \end{split}
\end{equation}
Noting that $1=b_0>b_1>b_2>\cdots>b_n>0$, $b_n\rightarrow 0$ $(n\rightarrow 0)$,
and $b_j^{-1}<b_{j+1}^{-1}$, making use of the induction assumption, we have
\begin{equation}\label{b6}
    \begin{split}
      \varphi^{m+1}
      \leq & C_0^{m+1}(1-b_1)(\varphi^0+b_{m-1}^{-1}\zeta)
      +\sum_{k=1}^{m-1}C_0^{m-k+1}(b_{k}-b_{k+1})(\varphi^0+b_{m-k-1}^{-1}\zeta)\\
      &+C_0b_m(\varphi^0+b_m^{-1}\zeta)\\
      \leq & C_0^{m+1}(\varphi^0+b_{m}^{-1}\zeta)
      \Big[(1-b_1)+\sum_{k=1}^{m-1}(b_{k}-b_{k+1})+b_m\Big]\\
      = & C_0^{m+1}(\varphi^0+b_{m}^{-1}\zeta).
    \end{split}
\end{equation}
Thus, \eqref{b3} is proved.

Next, following Reference \cite{Lin-Xu}, we see that $n^{-\alpha}b_{n-1}^{-1}\leq\dfrac{1}{1-\alpha}$, and obtain
\begin{align}\label{b7}
    \varphi^n
    \leq C_0^n(\varphi^0+b_{n-1}^{-1}\zeta)
    \leq C_0^n(\varphi^0+\dfrac{n^{\alpha}\tau^{\alpha}}{1-\alpha}\tau^{-\alpha}\zeta)
    \leq C_0^n(\varphi^0+\dfrac{T^{\alpha}}{1-\alpha}\tau^{-\alpha}\zeta).
\end{align}
Then we obtain the desired result.
\end{proof}

Next, we will give two identical relations of the bilinear forms $(\cdot,I_h^*\cdot)$ and $a(\cdot,I_h^*\cdot)$.

\begin{lemma}\label{lem3.7}
Let $\{\varphi^n\}_{n=0}^\infty$ be a function sequence on $V_h$,
then the following relation holds
\begin{equation}
   \begin{split}
      &a\Big(\varphi^n,I_h^*\Big(\sum_{k=0}^n \tilde{b}_k^n \varphi^k\Big)\Big)\\
      &\quad=\frac{1}{2}\Big[a(\varphi^n,I_h^*\varphi^n)
      +\sum_{k=0}^{n-1} \tilde{b}_k^n \Big(a(\varphi^k,I_h^*\varphi^k)
      -a(\varphi^n-\varphi^k,I_h^*(\varphi^n-\varphi^k))\Big)\Big]\\
      &\quad\ \ \ +\frac{1}{2}\sum_{k=0}^{n-1} \tilde{b}_k^n
      \Big[a(\varphi^n,I_h^*\varphi^k)-a(\varphi^k,I_h^*\varphi^n)\Big].
   \end{split}
\end{equation}
\end{lemma}

\begin{proof}
Noting the fact that $-\sum_{k=0}^{n-1}\tilde{b}_k^n=1$ and $b_n^n=1$, we have
\begin{equation}
  \begin{split}
     &a\Big(\varphi^n,I_h^*\Big(\sum_{k=0}^n \tilde{b}_k^n \varphi^k\Big)\Big)\\
     &\quad=a(\varphi^n,I_h^*\varphi^n)
     +\sum_{k=0}^{n-1} \tilde{b}_k^n a(\varphi^n,I_h^*\varphi^k)\\
     &\quad=\frac{1}{2}\Big[a(\varphi^n,I_h^*\varphi^n)
     +\sum_{k=0}^{n-1} \tilde{b}_k^n \Big(2a(\varphi^n,I_h^*\varphi^k)-a(\varphi^n,I_h^*\varphi^n)\Big)\Big]\\
     &\quad=\frac{1}{2}\Big[
      a(\varphi^n,I_h^*\varphi^n)
      +\sum_{k=0}^{n-1} \tilde{b}_k^n \Big(a(\varphi^k,I_h^*\varphi^k)
      -a(\varphi^n-\varphi^k,I_h^*(\varphi^n-\varphi^k))\Big)\Big]\\
       &\quad\ \ \ +\frac{1}{2}\sum_{k=0}^{n-1} \tilde{b}_k^n
      \Big[a(\varphi^n,I_h^*\varphi^k)-a(\varphi^k,I_h^*\varphi^n)\Big].
  \end{split}
\end{equation}
Thus, we complete the proof of this lemma.
\end{proof}

Applying Lemma \ref{lem3.1}, similar to the proof of Lemma \ref{lem3.7}, we can obtain the following identical relation.

\begin{lemma}\label{lem3.8}
Let $\{\varphi^n\}_{n=0}^\infty$ be a function sequence on $V_h$,
then the following relation holds
\begin{equation}\label{3.8}
   \begin{split}
     (\sum_{k=0}^n \tilde{b}_k^n \varphi^k,I_h^*\varphi^n)
   =\frac{1}{2}\Big[
   (\varphi^n,I_h^*\varphi^n)
   +\sum_{k=0}^{n-1} \tilde{b}_k^n (\varphi^k,I_h^*\varphi^k)
   -\sum_{k=0}^{n-1} \tilde{b}_k^n (\varphi^n-\varphi^k,I_h^*(\varphi^n-\varphi^k))
   \Big].
   \end{split}
\end{equation}
\end{lemma}

\section{Existence, Uniqueness and Stability Analysis}\label{sec4}

In this section, we will give the existence, uniqueness and stability results for the fully discrete FVE scheme \eqref{2.8}.

\begin{theorem}\label{thm4.1}
There exists a unique solution for the fully discrete FVE scheme \eqref{2.8}.
\end{theorem}

\begin{proof}
Let $M_Z^0$ be the number of the vertices in $Z_h^0$,
and $\{\Phi_i : i=1,2\cdots,M_Z^0\}$ be the abbreviated basis functions of the space $V_h$,
then $u_h^n\in V_h$ can be expressed as follows
\begin{align*}
     u_h^n(\bm{x})=\sum_{i=1}^{M_Z^0} u_i^n\Phi_i(\bm{x}).
\end{align*}

Substituting the above expression into the FVE scheme \eqref{2.8} (or the equivalent formulation \eqref{2.9}),
and taking $v_h=\Phi_j$ $(j=1,2,\cdots,M_Z^0)$,
then \eqref{2.8} (or \eqref{2.9}) can be rewritten as the following matrix form: find $U^n$ such that

\begin{equation}\label{4.1p}
    \frac{1}{\Gamma(2-\alpha)}B_1 U^n+\tau^\alpha B_2 U^n+\tau^\alpha B_3 U^n=\tau^\alpha F^n-\frac{1}{\Gamma(2-\alpha)}\sum_{k=0}^{n-1}\tilde{b}_k^n B_1 U^k,
\end{equation}
where
\begin{equation*}
   \begin{array}{ll}
    U^n=(u_1^n,u_2^n,\cdots,u_{M_Z^0}^n)^T,                    &B_1=((\Phi_i,I_h^*\Phi_j))_{i,j=1,\cdots,M_Z^0},           \\[1.5mm]
    B_2=(a(\Phi_i,I_h^*\Phi_j))_{i,j=1,\cdots,M_Z^0},    &B_3=((q\Phi_i,I_h^*\Phi_j))_{i,j=1,\cdots,M_Z^0}, \\[1.5mm]
    F^n=((f(t_n),I_h^*\Phi_j))_{j=1,\cdots,M_Z^0}^T.   &
\end{array}
\end{equation*}

Making use of Lemma \ref{lem3.1},
we can easily obtain that the matrix $B_1$ is symmetric positive definite.
Let $G=\frac{1}{\Gamma(2-\alpha)}B_1+\tau^\alpha B_2+\tau^\alpha B_3$, then \eqref{4.1p} can be rewritten as follows
\begin{equation}\label{4.2p}
    G U^n=\tau^\alpha F^n-\frac{1}{\Gamma(2-\alpha)}\sum_{k=0}^{n-1}\tilde{b}_k^n B_1 U^k.
\end{equation}

Next, we will prove $G$ is invertible.
Applying Lemma \ref{lem3.3}, for $\forall Y=(y_1,y_2,\cdots,y_{M_Z^0})^T\in R^{M_Z^0}\setminus \{\bm 0\}$, we have
\begin{align*}
   Y^T(B_2+B_3)Y=a(w_h,I_h^*w_h)+(qw_h,I_h^*w_h)\geq \mu_6\|w_h\|_1^2>0,
\end{align*}
where $w_h=\sum_{i=1}^{M_Z^0} y_i\Phi_i\neq 0$.
This means that $Y^T(B_2+B_3)Y$ (for $Y\in R^{M_Z^0}$) is a positive definite quadratic form generated by nonsymmetric matrix $(B_2+B_3)$.
Therefore, $Y^T G Y$ (for $Y\in R^{M_Z^0}$) is a positive definite quadratic form generated by nonsymmetric matrix $G$,
then we have that $G$ is invertible.
In fact, if $G$ is noninvertible, then the homogeneous linear equations $GY=0$ has nonzero solution $Y_0$,
thus, we have $Y_0^TGY_0=0$ which is in contradiction with the definition of positive definite quadratic form.
Hence, $G$ is invertible, then the linear equations \eqref{4.1p} have a unique solution.
This shows that the fully discrete FVE scheme \eqref{2.8} has a unique solution.
Then, we complete the proof.
\end{proof}

Next, we give the stability analysis for the fully discrete FVE scheme \eqref{2.8}.

%%%%%%%%%%%%%%%%%%%%%%%%%%%%%%%%%%%%%%%%%%%%%%%%%%%%%%%%%%
\begin{theorem}\label{thm4.2}
Let $\{u_h^n\}_{n=1}^N$ be the solution of the FVE system \eqref{2.8}, then there exists
a constant $C>0$ independent of $h$ and $\tau$ such that
\begin{align*}
     \|u_h^n\|\leq C(\|u_h^0\|+\sup_{t\in [0,T]}\|f(t)\|).
\end{align*}
Moreover, let $c_0>0$ be a constant,
then there exist constants $C>0$ and $\tau_0$ $(0<\tau_0<1)$ independent of $h$ and $\tau$ such that,
when $h\leq c_0\tau\leq c_0\tau_0$ and $h\leq h_1$, where $h_1=\min\{\frac{\mu_6}{\mu_5},1\}$, we have
\begin{align*}
     \|u_h^n\|_1 \leq Ce^{\frac{c_0T\mu_5}{2\mu_6}}(\|u_h^0\|_1+\sup_{t\in [0,T]}\|f(t)\|).
\end{align*}
\end{theorem}

\begin{proof}
Taking $v_h=u_h^n$ in \eqref{2.8}, we can obtain
\begin{align}\label{4.1}
      (D_{\tau}^{\alpha} u_h^n,I_h^*u_h^n)+a(u_h^n,I_h^*u_h^n)+(qu_h^n,I_h^*u_h^n)=(f^n,I_h^*u_h^n).
\end{align}
Make use of Lemma \ref{lem3.1} and Lemma \ref{lem3.3}, apply the Young inequality to obtain
\begin{align}\label{4.2}
      (D_{\tau}^{\alpha} u_h^n,I_h^*u_h^n)+\mu_6\|u_h^n\|_1^2
      \leq C\|f^n\|^2+\frac{\mu_6}{2}\|u_h^n\|^2.
\end{align}
Applying Lemma \ref{lem3.8} to rewrite $(D_{\tau}^{\alpha} u_h^n,I_h^*u_h^n)$, we have
\begin{equation}\label{4.3}
    \begin{split}
      (D_{\tau}^{\alpha} u_h^n,I_h^*u_h^n)=
      &\frac{\tau^{-\alpha}}{2\Gamma(2-\alpha)}(u_h^n,I_h^*u_h^n)
      +\frac{\tau^{-\alpha}}{2\Gamma(2-\alpha)}
      \sum_{k=0}^{n-1}\tilde{b}_k^n(u_h^k,I_h^*u_h^k) \\
      &-\frac{\tau^{-\alpha}}{2\Gamma(2-\alpha)}
      \sum_{k=0}^{n-1}\tilde{b}_k^n(u_h^n-u_h^k,I_h^*(u_h^n-u_h^k)).
    \end{split}
\end{equation}
Substituting \eqref{4.3} into \eqref{4.2}, and taking note of $\tilde{b}_k^n<0$ $(0\leq k\leq n-1)$, we obtain
\begin{equation}\label{4.4}
    \begin{split}
      &(u_h^n,I_h^*u_h^n)+\tau^{\alpha}\Gamma(2-\alpha)\mu_6\|u_h^n\|_1^2\\
      &\quad\leq -\sum_{k=0}^{n-1}\tilde{b}_k^n(u_h^k,I_h^*u_h^k)
      +C\tau^{\alpha}\Gamma(2-\alpha)  \sup_{t\in [0,T]}\|f(t)\|^2 .
    \end{split}
\end{equation}
Applying Lemma \ref{lem3.6} in \eqref{4.4}, we have
\begin{equation}\label{4.5}
   \begin{split}
       (u_h^n,I_h^*u_h^n)
       \leq C((u_h^0,I_h^*u_h^0)+\sup_{t\in [0,T]}\|f(t)\|^2).
   \end{split}
\end{equation}
Thus, applying Lemma \ref{lem3.1} in \eqref{4.5}, we obtain the following result
\begin{align}\label{4.6}
      \|u_h^n\|^2\leq C(\|u_h^0\|^2+\sup_{t\in [0,T]}\|f(t)\|^2).
\end{align}

Next, choosing $v_h=D_{\tau}^{\alpha} u_h^n$ in \eqref{2.8}, we get
\begin{equation}\label{4.7}
   \begin{split}
      (D_{\tau}^{\alpha} u_h^n,I_h^*D_{\tau}^{\alpha} u_h^n)+a(u_h^n,I_h^*D_{\tau}^{\alpha} u_h^n)
      =-(qu_h^n,I_h^*D_{\tau}^{\alpha} u_h^n)+(f^n,I_h^*D_{\tau}^{\alpha} u_h^n).
   \end{split}
\end{equation}
Noting that
$(D_{\tau}^{\alpha} u_h^n,I_h^*D_{\tau}^{\alpha} u_h^n)\geq \mu_1\|D_{\tau}^{\alpha} u_h^n\|^2$,
and applying the Young inequality in \eqref{4.7}, we have
\begin{equation}\label{4.8}
   \begin{split}
      \mu_1\|D_{\tau}^{\alpha} u_h^n\|^2+a(u_h^n,I_h^*D_{\tau}^{\alpha} u_h^n)
      \leq C(\|u_h^n\|^2+\|f^n\|^2)+\frac{\mu_1}{2}\|D_{\tau}^{\alpha} u_h^n\|^2.
   \end{split}
\end{equation}
Apply Lemma \ref{lem3.7} in \eqref{4.8} to obtain
\begin{equation}\label{4.9}
    \begin{split}
      &\mu_1\|D_{\tau}^{\alpha} u_h^n\|^2
      +\frac{\tau^{-\alpha}}{2\Gamma(2-\alpha)}a(u_h^n,I_h^*u_h^n)
      +\frac{\tau^{-\alpha}}{2\Gamma(2-\alpha)}\sum_{k=0}^{n-1}\tilde{b}_k^n a(u_h^k,I_h^*u_h^k)\\
      &-\frac{\tau^{-\alpha}}{2\Gamma(2-\alpha)}\sum_{k=0}^{n-1}\tilde{b}_k^n a(u_h^n-u_h^k,I_h^*(u_h^n-u_h^k))\\
      &\quad\leq C(\|u_h^n\|^2+\|f^n\|^2)+\frac{\mu_1}{2}\|D_{\tau}^{\alpha} u_h^n\|^2\\
      &\quad\ \ \ -\frac{\tau^{-\alpha}}{2\Gamma(2-\alpha)} \sum_{k=0}^{n-1}\tilde{b}_k^n\Big[a(u_h^n,I_h^*u_h^k)-a(u_h^k,I_h^*u_h^n)\Big].
    \end{split}
\end{equation}
Making use of Lemma \ref{lem3.2} and Lemma \ref{lem3.3}, we have
\begin{equation}\label{4.10}
   \begin{split}
      |a(u_h^n,I_h^*u_h^k)-a(u_h^k,I_h^*u_h^n)|
      &\leq \mu_5 h \|u_h^n\|_1 \|u_h^k\|_1\\
      &\leq \frac{\mu_5}{2\mu_6}h [a(u_h^n,I_h^*u_h^n)+a(u_h^k,I_h^*u_h^k)].
   \end{split}
\end{equation}
Substituting \eqref{4.6} and \eqref{4.10} into \eqref{4.9},
multiplying \eqref{4.9} by $\tau^\alpha$, and noting that $\tilde{b}_k^n<0$ $(0\leq k\leq n-1)$,
we have
\begin{equation}\label{4.11}
    \begin{split}
       (1-\frac{\mu_5}{2\mu_6}h)a(u_h^n,I_h^*u_h^n)
       \leq -(1+\frac{\mu_5}{2\mu_6}h)\sum_{k=0}^{n-1} \tilde{b}_k^n a(u_h^k,I_h^*u_h^k)
       +C\tau^\alpha(\|u_h^0\|^2+\sup_{t\in [0,T]}\|f(t)\|^2).
   \end{split}
\end{equation}
Setting $h_1=\min\{\frac{\mu_6}{\mu_5},1\}$, when $h\leq h_1$, we have $(1-\frac{\mu_5}{2\mu_6}h)\geq \frac{1}{2}$ and
\begin{equation}\label{4.12}
    \begin{split}
       a(u_h^n,I_h^*u_h^n)
       \leq -\frac{1+\frac{\mu_5}{2\mu_6}h}{1-\frac{\mu_5}{2\mu_6}h}
       \sum_{k=0}^{n-1} \tilde{b}_k^n a(u_h^k,I_h^*u_h^k)
       +C\tau^\alpha(\|u_h^0\|^2+\sup_{t\in [0,T]}\|f(t)\|^2).
   \end{split}
\end{equation}
Applying Lemma \ref{lem3.6} in \eqref{4.12}, we have
\begin{equation}\label{4.13}
    \begin{split}
       a(u_h^n,I_h^*u_h^n)
       \leq C\Big(\frac{1+\frac{\mu_5}{2\mu_6}h}{1-\frac{\mu_5}{2\mu_6}h}\Big)^n
       (a(u_h^0,I_h^*u_h^0)+\|u_h^0\|^2+\sup_{t\in [0,T]}\|f(t)\|^2).
   \end{split}
\end{equation}
Let $c_0>0$ be a constant. Selecting $h$ and $\tau$ to satisfy $h\leq c_0\tau$ in \eqref{4.13}, and applying Lemma \ref{lem3.3}, we have
\begin{equation}\label{4.14}
    \begin{split}
       \|u_h^n\|_1^2
       \leq C\Big(\frac{1+\frac{c_0\mu_5}{2\mu_6}\tau}{1-\frac{c_0\mu_5}{2\mu_6}\tau}\Big)^{\frac{T}{\tau}}
       (\|u_h^0\|_1^2+\sup_{t\in [0,T]}\|f(t)\|^2).
   \end{split}
\end{equation}
Note that
\begin{align}\label{4.15}
   \lim_{\tau\rightarrow 0} \Big(\frac{1+\frac{c_0\mu_5}{2\mu_6}\tau}{1-\frac{c_0\mu_5}{2\mu_6}\tau}\Big)^{\frac{T}{\tau}}
   =e^{\frac{c_0T\mu_5}{\mu_6}},
\end{align}
then, there exists a constant $\tau_0$ $(0<\tau_0<1)$, when $\tau\leq\tau_0$, it follows that
\begin{equation}\label{4.16}
    \begin{split}
       \|u_h^n\|_1^2
       \leq Ce^{\frac{c_0T\mu_5}{\mu_6}}(\|u_h^0\|_1^2+\sup_{t\in [0,T]}\|f(t)\|^2).
   \end{split}
\end{equation}
Thus, we complete the proof of the stability.
\end{proof}

\begin{remark}\label{rmk4.1}
   From Theorem \ref{thm4.2}, it is easy to see that $\|u_h^n\|$ is unconditionally stable,
   and $\|u_h^n\|_1$ is stable when the conditions in the Theorem \eqref{thm4.2} are established,
   because the bilinear $a(\dot,I_h^*\dot)$ does not necessarily satisfy symmetry.
   When the coefficient $\mathcal{A}(\bm{x})$ is a symmetry and positive definite constant matrix,
   following Reference \cite{Ewing-Lazarov-Lin},
   we can know that $a(\dot,I_h^*\dot)$ satisfy symmetry.
   And under this condition, we can obtain that $\|u_h^n\|_1$ is also unconditionally stable.
\end{remark}

\section{\textit{A Priori} Error Estimates}

In order to obtain the error estimates for the fully discrete FVE scheme \eqref{2.8},
we introduce an elliptic projection operator $P_h: H_0^1(\Omega)\cap H^2(\Omega)\rightarrow V_h$,
which is defined by the following
\begin{align}\label{5.1}
      a(u-P_hu,I_h^*v_h)=0, \ \forall v_h\in V_h.
\end{align}
Following References \cite{Li-Chen-Wu}, the above projection operator $P_h$ satisfies the following estimates.
\begin{lemma}\label{lem5.1}
There exists a constant $C>0$ independent of $h$ and $\tau$ such that
\begin{align}
      &\|u-P_hu\|_1\leq Ch|u|_2,\ \forall u\in H_0^1(\Omega)\cap H^2(\Omega),\label{5.2}\\
      &\|u-P_hu\|\leq Ch^2\|u\|_{3,p},\ \forall u\in H_0^1(\Omega)\cap W^{3,p}(\Omega),\ p>1.\label{5.3}
\end{align}
\end{lemma}

Next, we give the main results in this paper about the error estimates.

\begin{theorem}\label{thm5.1}
Let $u$ and $u_h^n$ be the solutions of system \eqref{2.0} and the FVE scheme \eqref{2.8}, respectively.
Assume that $u_h^0=P_hu_0$, then there exists a constant $C>0$ independent of $h$ and $\tau$ such that
\begin{align}
      &\max_{1\leq n\leq N}\|u(t_n)-u_h^n\|\leq C(h^2+\tau^{2-\alpha}),\label{5.4}\\
      &\max_{1\leq n\leq N}\|u(t_n)-u_h^n\|_1\leq C(h+\tau^{-\frac{\alpha}{2}}(h^2+\tau^{2-\alpha})).\label{5.5}
\end{align}
\end{theorem}

\begin{proof}
we write $u(t_n)-u_h^n=\rho^n+\theta^n$,
where $\rho^n=u(t_n)-P_hu(t_n)$, $\theta^n=P_hu(t_n)-u_h^n$,
and $P_h$ is the elliptic projection operator defined in \eqref{5.1}.
According to Lemma \ref{lem5.1}, we only need to estimate $\theta^n$.
Making using of our definitions, we can have the error equation for $\theta^n$ as follows
\begin{equation}\label{5.6}
    \begin{split}
      &(D_{\tau}^{\alpha}\theta^n,I_h^*v_h)+a(\theta^n,I_h^*v_h)+(p\theta^n,I_h^*v_h)\\
      &\quad=-(D_{\tau}^{\alpha}\rho^n,I_h^*v_h)
      -(q\rho^n,I_h^*v_h)-(R_t^n(\bm{x}),I_h^*v_h),\ \forall v_h\in V_h.
    \end{split}
\end{equation}

Now, choosing $v_h=\theta^n$ in \eqref{5.6}, we obtain
\begin{equation}\label{5.7}
    \begin{split}
       &(D_{\tau}^{\alpha}\theta^n,I_h^*\theta^n)+a(\theta^n,I_h^*\theta^n)
       +(q\theta^n,I_h^*\theta^n)\\
       &\quad=-(D_{\tau}^{\alpha}\rho^n,I_h^*\theta^n)-(q\rho^n,I_h^*\theta^n)-(R_t^n(\bm{x}),I_h^*\theta^n).
    \end{split}
\end{equation}
By virtue of Lemma \ref{lem3.1} and the Young inequality in \eqref{5.7}, we have
\begin{equation}\label{5.8}
    \begin{split}
      (D_{\tau}^{\alpha}\theta^n,I_h^*\theta^n)+a(\theta^n,I_h^*\theta^n)
      \leq C(\|D_{\tau}^{\alpha}\rho^n\|^2+\|\rho^n\|^2+\|R_t^n(\bm{x})\|^2)
      +\frac{\mu_6}{2}\|\theta^n\|^2.
    \end{split}
\end{equation}
Apply Lemma \ref{lem3.3} and Lemma \ref{lem3.8} in \eqref{5.8} to obtain
\begin{equation}\label{5.9}
    \begin{split}
       &\frac{\tau^{-\alpha}}{2\Gamma(2-\alpha)}(\theta^n,I_h^*\theta^n)
      +\frac{\tau^{-\alpha}}{2\Gamma(2-\alpha)}
      \sum_{k=0}^{n-1}\tilde{b}_k^n(\theta^k,I_h^*\theta^k) \\
      &-\frac{\tau^{-\alpha}}{2\Gamma(2-\alpha)}
      \sum_{k=0}^{n-1}\tilde{b}_k^n(\theta^n-\theta^k,I_h^*(\theta^n-\theta^k))
      +\frac{\mu_6}{2}\|\theta^n\|_1^2\\
      &\quad\leq C(\|D_{\tau}^{\alpha}\rho^n\|^2+\|\rho^n\|^2+\|R_t^n(\bm{x})\|^2).
    \end{split}
\end{equation}
Multiplying \eqref{5.9} by $2\Gamma(2-\alpha)\tau^\alpha$, and noting that
$-\sum_{k=0}^{n-1}\tilde{b}_k^n(\theta^n-\theta^k,I_h^*(\theta^n-\theta^k))\geq 0$,
we have
\begin{equation}\label{5.10}
    \begin{split}
      &(\theta^n,I_h^*\theta^n)+\tau^{\alpha}\Gamma(2-\alpha)\mu_6\|\theta^n\|_1^2\\
      &\quad\leq -\sum_{k=0}^{n-1}\tilde{b}_k^n(\theta^k,I_h^*\theta^k)
      +C\tau^{\alpha}\Gamma(2-\alpha)(\|D_{\tau}^{\alpha}\rho^n\|^2+\|\rho^n\|^2+\|R_t^n(\bm{x})\|^2).
    \end{split}
\end{equation}
In order to estimate $\|D_{\tau}^{\alpha}\rho^n\|^2$ and $\|\rho^n\|^2$ in \eqref{5.10}, we apply Lemma \ref{lem5.1} to obtain
\begin{equation}\label{5.11}
   \begin{split}
      &\|\rho^n\|^2=\|u(t_n)-P_hu(t_n)\|^2\leq C h^4\|u(t_n)\|_{W^{3,p}}\leq Ch^4\|u\|_{L^\infty(W^{3,p})},\\
      &\|\partial _t \rho^{n-k}\|=\|\frac{1}{\tau}\int_{t_{n-k-1}}^{t_{n-k}}\rho_t \d t\|
      \leq h^2 \|u_t\|_{L^\infty(W^{3,p})},\ p>1.
   \end{split}
\end{equation}
Noting that $D_{\tau}^{\alpha}\rho^n=\frac{\tau^{1-\alpha}}{\Gamma(2-\alpha)}\sum_{k=0}^{n-1}b_k\partial _t \rho^{n-k}$
and $\sum_{k=0}^{n-1}b_k=n^{1-\alpha}$, we have
\begin{align}\label{5.12}
   \|D_{\tau}^{\alpha}\rho^n\|^2
   \leq \Big(\frac{\tau^{1-\alpha}}{\Gamma(2-\alpha)}n^{1-\alpha}h^2 \|u_t\|_{L^\infty(W^{3,p})}\Big)^2
   \leq CT^{2-2\alpha}h^4 \|u_t\|_{L^\infty(W^{3,p})}^2,\ p>1.
\end{align}
Substituting \eqref{5.11} and \eqref{5.12} into \eqref{5.10}, and applying Lemma \ref{lem3.4}, we have
\begin{equation}\label{5.13}
    \begin{split}
      &(\theta^n,I_h^*\theta^n)+\tau^{\alpha}\Gamma(2-\alpha)\mu_6\|\theta^n\|_1^2\\
     &\quad\leq -\sum_{k=0}^{n-1}\tilde{b}_k^n(\theta^k,I_h^*\theta^k)
      +C\tau^{\alpha}(h^4+\tau^{2(2-\alpha)}+\tau^4).
    \end{split}
\end{equation}
Note that $\theta^0=0$, apply Lemma \ref{lem3.5} in \eqref{5.13} to obtain
\begin{equation}\label{5.14}
    \begin{split}
      (\theta^n,I_h^*\theta^n)+\tau^{\alpha}\Gamma(2-\alpha)\mu_6\|\theta^n\|_1^2
       \leq C(h^4+\tau^{2(2-\alpha)}+\tau^4).
    \end{split}
\end{equation}
By virtue of Lemma \ref{lem3.1}, we have
\begin{equation}\label{5.15}
    \begin{split}
      \mu_1\|\theta^n\|^2+\tau^{\alpha}\Gamma(2-\alpha)\mu_6\|\theta^n\|_1^2
      \leq C(h^4+\tau^{2(2-\alpha)}+\tau^4).
    \end{split}
\end{equation}
Finally, apply Lemma \ref{lem5.1} with \eqref{5.15} to complete the proof.
\end{proof}

\begin{remark}\label{rmk5.1}
   From Theorem \ref{thm5.1},
we can see that $\|u(t_n)-u_h^n\|$ has been given the optimal \textit{a priori} error estimate based on $L1$-formula,
and the estimate in time for $\|u(t_n)-u_h^n\|_1$ reduces $\frac{\alpha}{2}$ by comparing with the optimal estimate.
So we need to find and use other estimate methods.
\end{remark}

Next, we try to give a better estimate for $\|u(t_n)-u_h^n\|_1$.

\begin{theorem}\label{thm5.2}
Let $u$ and $u_h^n$ be the solutions of system \eqref{2.0} and the FVE scheme \eqref{2.8}, respectively.
Assume that $u_h^0=P_hu_0$.
Let $c_0>0$ be a constant,
then there exist two constants $C>0$ and $\tau_0$ $(0<\tau_0<1)$ independent of $h$ and $\tau$ such that,
when $h\leq c_0\tau\leq c_0\tau_0$ and $h\leq h_1$, where $h_1=\min\{\frac{\mu_6}{\mu_5},1\}$, we have
\begin{align*}
      \max_{1\leq n\leq N}\|u(t_n)-u_h^n\|_1\leq C(h+e^{\frac{c_0T\mu_5}{2\mu_6}}(h^2+\tau^{2-\alpha})).
\end{align*}
\end{theorem}

\begin{proof}
First, we take $v_h=D_{\tau}^{\alpha} \theta^n$ in \eqref{5.6} to obtain
\begin{equation}\label{5.16}
   \begin{split}
      &(D_{\tau}^{\alpha} \theta^n,I_h^*D_{\tau}^{\alpha} \theta^n)+a(\theta^n,I_h^*D_{\tau}^{\alpha} \theta^n)\\
      &\quad=-(D_{\tau}^{\alpha} \rho^n,I_h^*D_{\tau}^{\alpha} \theta^n)
      -(q\rho^n,I_h^*D_{\tau}^{\alpha} \theta^n)-(q\theta^n,I_h^*D_{\tau}^{\alpha} \theta^n)
      -(R_t^n(\bm{x}),I_h^*D_{\tau}^{\alpha} \theta^n).
   \end{split}
\end{equation}
Noting that
$(D_{\tau}^{\alpha} \theta^n,I_h^*D_{\tau}^{\alpha} \theta^n)\geq \mu_1\|D_{\tau}^{\alpha} \theta^n\|^2$,
applying Lemma \ref{lem3.1} and the Young inequality, we have
\begin{equation}\label{5.17}
   \begin{split}
      &\mu_1\|D_{\tau}^{\alpha} \theta^n\|^2+a(\theta^n,I_h^*D_{\tau}^{\alpha} \theta^n)\\
      &\quad\leq C(\|D_{\tau}^{\alpha} \rho^n\|^2+\|\rho^n\|^2+\|\theta^n\|^2+\|R_t^n(\bm{x})\|^2)+
      \frac{\mu_1}{2}\|D_{\tau}^{\alpha} \theta^n\|^2.
   \end{split}
\end{equation}
Apply Lemma \ref{lem3.7} in \eqref{5.17} to obtain
\begin{equation}\label{5.18}
    \begin{split}
       &\frac{\mu_1}{2}\|D_{\tau}^{\alpha} \theta^n\|^2
       +\frac{\tau^{-\alpha}}{2\Gamma(2-\alpha)}a(\theta^n,I_h^*\theta^n)
      +\frac{\tau^{-\alpha}}{2\Gamma(2-\alpha)}
      \sum_{k=0}^{n-1}\tilde{b}_k^n a(\theta^k,I_h^*\theta^k) \\
      &-\frac{\tau^{-\alpha}}{2\Gamma(2-\alpha)} \sum_{k=0}^{n-1}\tilde{b}_k^n a(\theta^n-\theta^k,I_h^*(\theta^n-\theta^k))\\
      &\quad\leq C(\|D_{\tau}^{\alpha} \rho^n\|^2+\|\rho^n\|^2+\|\theta^n\|^2+\|R_t^n(\bm{x})\|^2)\\
      &\quad\ \ \ -\frac{\tau^{-\alpha}}{2\Gamma(2-\alpha)}
      \sum_{k=0}^{n-1} \tilde{b}_k^n[a(\theta^n,I_h^*\theta^k)-a(\theta^k,I_h^*\theta^n)].
    \end{split}
\end{equation}
Making use of Lemma \ref{lem3.2} and Lemma \ref{lem3.3}, we have
\begin{align}\label{5.19}
   |a(\theta^n,I_h^*\theta^k)-a(\theta^k,I_h^*\theta^n)|\leq \mu_5 h \|\theta^n\|_1 \|\theta^k\|_1
   \leq \frac{\mu_5}{2\mu_6}h[a(\theta^n,I_h^*\theta^n)+ a(\theta^k,I_h^*\theta^k)].
\end{align}
Noting that $-\frac{\tau^{-\alpha}}{2\Gamma(2-\alpha)}\sum_{k=0}^{n-1}\tilde{b}_k^n a(\theta^n-\theta^k,I_h^*(\theta^n-\theta^k))\geq 0$,
and substituting \eqref{5.19} into \eqref{5.18}, we obtain
\begin{equation}\label{5.20}
    \begin{split}
       &\mu_1\tau^\alpha\Gamma(2-\alpha)\|D_{\tau}^{\alpha} \theta^n\|^2
       +(1-\frac{\mu_5}{2\mu_6}h)a(\theta^n,I_h^*\theta^n)\\
       &\quad\leq -(1+\frac{\mu_5}{2\mu_6}h)\sum_{k=0}^{n-1} \tilde{b}_k^n a(\theta^k,I_h^*\theta^k)\\
       &\quad\ \ \ +C\tau^\alpha\Gamma(2-\alpha)
       (\|D_{\tau}^{\alpha} \rho^n\|^2+\|\rho^n\|^2+\|\theta^n\|^2+\|R_t^n(\bm{x})\|^2).
   \end{split}
\end{equation}
Substituting \eqref{5.11}, \eqref{5.12} and \eqref{5.15} into \eqref{5.20}, and applying Lemma \ref{lem3.4},
we have
\begin{equation}\label{5.21}
    \begin{split}
       &\mu_1\tau^\alpha\Gamma(2-\alpha)\|D_{\tau}^{\alpha} \theta^n\|^2
       +(1-\frac{\mu_5}{2\mu_6}h)a(\theta^n,I_h^*\theta^n)\\
       &\quad\leq -(1+\frac{\mu_5}{2\mu_6}h)\sum_{k=0}^{n-1} \tilde{b}_k^n a(\theta^k,I_h^*\theta^k)
       +C\tau^\alpha (h^4+\tau^{2(2-\alpha)}+\tau^4).
   \end{split}
\end{equation}
Similar to the proof process of Theorem \ref{thm4.2}, when $h\leq h_1$, where $h_1=\min\{\frac{\mu_6}{\mu_5},1\}$, we have
\begin{equation}\label{5.22}
    \begin{split}
       a(\theta^n,I_h^*\theta^n)
       \leq -\frac{1+\frac{\mu_5}{2\mu_6}h}{1-\frac{\mu_5}{2\mu_6}h}\sum_{k=0}^{n-1} \tilde{b}_k^n a(\theta^k,I_h^*\theta^k)
       +C\tau^\alpha (h^4+\tau^{2(2-\alpha)}+\tau^4).
   \end{split}
\end{equation}
Applying Lemma \ref{lem3.6} in \eqref{5.22}, noting that $\theta^0=0$, we easily get
\begin{equation}\label{5.23}
    \begin{split}
       a(\theta^n,I_h^*\theta^n)
       \leq C\Big(\frac{1+\frac{\mu_5}{2\mu_6}h}{1-\frac{\mu_5}{2\mu_6}h}\Big)^n (h^4+\tau^{2(2-\alpha)}+\tau^4).
   \end{split}
\end{equation}
Let $c_0>0$ be a constant. Selecting $h$ and $\tau$ to satisfy $h\leq c_0 \tau$, we can obtain
\begin{equation}\label{5.24}
    \begin{split}
       a(\theta^n,I_h^*\theta^n)
       \leq C\Big(\frac{1+\frac{c_0\mu_5}{2\mu_6}\tau}{1-\frac{c_0\mu_5}{2\mu_6}\tau}\Big)^{\frac{T}{\tau}} (h^4+\tau^{2(2-\alpha)}+\tau^4).
   \end{split}
\end{equation}
By virtue of \eqref{4.15}, there exists a constant $\tau_0$ $(0<\tau_0<1)$, when $\tau\leq\tau_0$, it follows that
\begin{equation}\label{5.26}
    \begin{split}
       a(\theta^n,I_h^*\theta^n)
       \leq Ce^{\frac{c_0T\mu_5}{\mu_6}}(h^4+\tau^{2(2-\alpha)}+\tau^4).
   \end{split}
\end{equation}
Finally, apply Lemma \ref{lem3.3} and Lemma \ref{lem5.1} with \eqref{5.26} to complete the proof.
\end{proof}

\begin{remark}\label{rmk5.2}
   From Theorem \ref{thm5.2},
we can see that the optimal \textit{a priori} error estimate for $\|u(t_n)-u_h^n\|_1$ is obtained
when $h$ and $\tau$ satisfy $h\leq c_0\tau\leq c_0\tau_0$.
When $h$ is sufficiently small, we have $\|u(t_n)-u_h^n\|_1\leq C(\tau^{2-\alpha})$, which is validated in the Example 1 in Section \ref{sec6}.
\end{remark}
\begin{remark}\label{rmk5.3}
   Similar to Remark \ref{rmk4.1},
   When the coefficient $\mathcal{A}(\bm{x})$ is a symmetry and positive definite constant matrix,
   in the analysis and results of Theorem \ref{5.2}, we can remove the conditions $h\leq c_0\tau\leq c_0\tau_0$ and $h\leq h_1$.
\end{remark}

\section{Numerical Examples}\label{sec6}

In this section, we will give two examples to examine the feasibility and effectiveness of the proposed FVE scheme.

{\bf Example 1.}\ \  We consider the equations \eqref{1.1} in one-dimensional space regions as follows
\begin{align}\label{6.1}
   \left\{
   \begin{array}{ll}
        {}_0^CD_t^{\alpha}u(x,t)-\frac{\partial}{\partial x} (\tilde{a}(x)\frac{\partial u(x,t)}{\partial x})+q(x)u(x,t)
        =f(x,t),& (x,t)\in \Omega\times J,\\
        u(a,t)=u(b,t)=0, &  t\in \bar{J}, \\
        u(x,0)=u_0(x),& x\in \bar{\Omega}.
   \end{array}\right.
\end{align}
where $\Omega=(a,b)\subset R^1$, $J=(0,T]$ is the time interval with $0<T<\infty$.
The functions $\tilde{a}(x)$, $q(x)$, $f(x,t)$ and $u_0(x)$
are smooth enough, and $q(x)\geq 0$, $\forall x\in \bar{\Omega}$.
Suppose that there exist two constants $\tilde{a}_0$ and $\tilde{a}_1$ such that
$0<\tilde{a}_0\leq \tilde{a}(x)\leq \tilde{a}_1$.
Following Reference \cite{Li-Chen-Wu}, we construct the primal mesh $\mathcal{\widetilde{T}}_h$
and dual mesh $\mathcal{\widetilde{T}}_h^*$,
and take the spaces $\tilde{V}_h$ and $\tilde{V}_h^*$ as the \textit{trial} function space and \textit{test} function space, respectively, where
\begin{align*}
     &\tilde{V}_h=\{w_h\in H_0^1(\Omega): w_h\in P_1(A),\ \forall A\in \mathcal{\widetilde{T}}_h\},\\
     &\tilde{V}_h^*=\{v_h\in L^2(\Omega): v_h|_{A^*}\in P_0(A^*),\ \forall A^*\in \mathcal{\widetilde{T}}_h^*,\ v_h|_{\partial\Omega}=0\}.
\end{align*}
As in Reference \cite{Li-Chen-Wu}, we also define operator $\tilde{I}_h^*: C(\Omega)\rightarrow \tilde{V}_h^*$ by
\begin{align*}
     \tilde{I}_h^*w_h=\sum_{i=1}^{\tilde{M}_T^0}w_h(x_i)\chi_{A_i^*},\ \forall w_h\in \tilde{V}_h^*,
\end{align*}
where $\tilde{M}_T^0$ is the number of the interior nodes,
and $\chi_{A_i^*}$ is the characteristic function of a set $A_i^*$ with $A_i^*\in \mathcal{\widetilde{T}}_h^*$.
Making use of the operator $\tilde{I}_h^*$,
we can also present the FVE scheme and obtain the corresponding theoretical results as Theorems \ref{thm4.1}-\ref{thm4.2} and Theorems \ref{thm5.1}-\ref{thm5.2}.
Here, we will not repeat these processes.
\begin{table}
\caption{Error results for $\max\limits_{n}\|u(t_n)-u_h^n\|$ with $h=2.5E-04$ in Example 1.}
\centering
{\begin{tabular}{lllll} \toprule
  $\alpha$ & $\tau_1=1/10$ & $\tau_2=1/20$ & $\tau_3=1/40$ & $\tau_4=1/80$\\
 \midrule
 0.1  & 1.37013229E-05 & 3.93247726E-06 & 1.06811885E-06 & 2.86820027E-07\\
 Rate &                & 1.80080487     & 1.88036624     & 1.89685451\\
 0.3  & 6.53433449E-05 & 2.09751531E-05 & 6.60479132E-06 & 2.03573159E-06\\
 Rate &                & 1.63935897     & 1.66709645     & 1.69796563\\
 0.5  & 1.77931804E-04 & 6.41016920E-05 & 2.28988168E-05 & 8.10013858E-06\\
 Rate &                & 1.47289006     & 1.48508938     & 1.49925456\\
 0.7  & 4.16102111E-04 & 1.70098563E-04 & 6.93000954E-05 & 2.81419417E-05\\
 Rate &                & 1.29056665     & 1.29544171     & 1.30013547\\
 0.9  & 9.03244861E-04 & 4.21869751E-04 & 1.96878771E-04 & 9.18133390E-05\\
 Rate &                & 1.09831949     & 1.09949009     & 1.10053188\\
 \bottomrule
\end{tabular}}
\label{tab1}
\end{table}
\begin{table}
\caption{Error results for $\max\limits_{n}\|u(t_n)-u_h^n\|_1$ with $h=2.5E-04$ in Example 1.}
\centering
{\begin{tabular}{lllll} \toprule
$\alpha$ & $\tau_1=1/10$ & $\tau_2=1/20$ & $\tau_3=1/40$ & $\tau_4=1/80$ \\
 \midrule
 0.1  & 8.40509300E-05 & 2.46792171E-05 & 7.15127494E-06 & 2.05289204E-06\\
 Rate &                & 1.76796716     & 1.78702425     & 1.80054271\\
 0.3  & 3.97616219E-04 & 1.28155606E-04 & 4.08865782E-05 & 1.29439996E-05\\
 Rate &                & 1.63348002     & 1.64819735     & 1.65934386\\
 0.5  & 1.08020056E-03 & 3.89615121E-04 & 1.39673166E-04 & 4.98620848E-05\\
 Rate &                & 1.47117763     & 1.47999479     & 1.48603977\\
 0.7  & 2.52162418E-03 & 1.03119069E-03 & 4.20560268E-04 & 1.71241407E-04\\
 Rate &                & 1.29004214     & 1.29392668     & 1.29628096\\
 0.9  & 5.46439531E-03 & 2.55250980E-03 & 1.19159979E-03 & 5.56103741E-04\\
 Rate &                & 1.09814536     & 1.09901672     & 1.09947383\\
 \bottomrule
\end{tabular}}
\label{tab2}
\end{table}

In this example, we take $\Omega=(0,1)$, $T=1$, $\tilde{a}(x)=1+2x^2$, $q(x)=1+x^2$, $u_0(x)=0$
and
$$f(x,t)=(\frac{2}{\Gamma(3-\alpha)}t^{2-\alpha}+t^2(1+x^2)+4\pi^2t^2(1+2x^2))\sin(2\pi x)-8\pi t^2x\cos(2\pi x).$$
Thus we can obtain the analytical solution $u(x,t)=t^2\sin(2\pi x)$.

We give the numerical results with some different parameters $\alpha=0.1$, $0.3$, $0.5$, $0.7$, $0.9$ in Tables \ref{tab1}-\ref{tab4}.
In order to test the time convergence rates,
we choose the spatial step length $h=2.5E-04$ and the time step length $\tau=1/10, 1/20, 1/40, 1/80$,
and give the error behaviors for $u$ with $L^2(\Omega)$-norm (in Table \ref{tab1})
and $H^1(\Omega)$-norm (in Table \ref{tab2}).
We can find that the time convergence rates in these two norms are approximate to $2-\alpha$,
which are consistent with the convergence results in Theorems \ref{thm5.1}-\ref{thm5.2}.
Moreover, choosing different parameter $\alpha=0.1, 0.3, 0.5, 0.7, 0.9$, fixing the time step length $\tau=1E-03$,
regardless of the condition $h\leq c_0\tau$ in Theorem \ref{thm5.2},
we take the spatial step length $h=1/10, 1/20, 1/40, 1/80$,
and give error behaviors in Tables \ref{tab3}-\ref{tab4} to test the spatial convergence rates,
and find that the spatial convergence rates for $u$ with $L^2(\Omega)$-norm
and $H^1(\Omega)$-norm are approximate to $2$.
This means that the convergence can still be satisfied in the actual calculation,
even if the condition $h\leq c_0\tau$ in Theorem \ref{thm5.2} is not established.
\begin{table}
\caption{Error results for $\max\limits_{n}\|u(t_n)-u_h^n\|$ with $\tau=1E-03$ in Example 1.}
\centering
{\begin{tabular}{lllll} \toprule
$\alpha$ & $h_1=1/10$ & $h_2=1/20$ & $h_3=1/40$ & $h_4=1/80$\\
 \midrule
 0.1  & 2.30102485E-02 & 5.71793833E-03 & 1.42733291E-03 & 3.56696835E-04\\
 Rate &                & 2.00870960     & 2.00217319     & 2.00055155\\
 0.3  & 2.29764033E-02 & 5.70898201E-03 & 1.42504516E-03 & 3.56104256E-04\\
 Rate &                & 2.00884756     & 2.00222587     & 2.00063606\\
 0.5  & 2.29407203E-02 & 5.69945327E-03 & 1.42252431E-03 & 3.55364857E-04\\
 Rate &                & 2.00901525     & 2.00237022     & 2.00108039\\
 0.7  & 2.29039730E-02 & 5.68916680E-03 & 1.41932930E-03 & 3.53972818E-04\\
 Rate &                & 2.00930859     & 2.00300804     & 2.00349886\\
 0.9  & 2.28652818E-02 & 5.67586448E-03 & 1.41283646E-03 & 3.49217385E-04\\
 Rate &                & 2.01024665     & 2.00624566     & 2.01639719\\
 \bottomrule
\end{tabular}}
\label{tab3}
\end{table}
\begin{table}
\caption{Error results for $\max\limits_{n}\|u(t_n)-u_h^n\|_1$ with $\tau=1E-03$ in Example 1.}
\centering
{\begin{tabular}{lllll} \toprule
$\alpha$ & $h_1=1/10$ & $h_2=1/20$ & $h_3=1/40$ & $h_4=1/80$\\
 \midrule
 0.1  & 2.31685939E-02 & 5.76008882E-03 & 1.43803366E-03 & 3.59383963E-04\\
 Rate &                & 2.00800753     & 2.00199361     & 2.00049951\\
 0.3  & 2.31628844E-02 & 5.75865046E-03 & 1.43767150E-03 & 3.59291324E-04\\
 Rate &                & 2.00801226     & 2.00199669     & 2.00050806\\
 0.5  & 2.31594188E-02 & 5.75781483E-03 & 1.43746182E-03 & 3.59237288E-04\\
 Rate &                & 2.00800575     & 2.00199776     & 2.00051463\\
 0.7  & 2.31586206E-02 & 5.75772620E-03 & 1.43748501E-03 & 3.59322919E-04\\
 Rate &                & 2.00797823     & 2.00195227     & 2.00019405\\
 0.9  & 2.31608778E-02 & 5.75888309E-03 & 1.43853377E-03 & 3.61460058E-04\\
 Rate &                & 2.00782899     & 2.00118994     & 1.99269095\\
 \bottomrule
\end{tabular}}
\label{tab4}
\end{table}
%%%%%%%%%%%%%%%%%%%%%%%%%%%%%%%%%%%%%%%%%%%%%%%%%%%%%%%%%%%%%%%%%%%%%%%%%%%%%%%%%%%%%%%%%%%%%%%%%%%%%%%%%%%%%%%%%%%%%
\begin{table}
\caption{Error results for $\max\limits_{n}\|u(t_n)-u_h^n\|$ in Example 2 with $\alpha=0.1$.}
\centering
{\begin{tabular}{lllll} \toprule
& $s_1=1/5$ & $s_2=1/10$ & $s_3=1/20$ & $s_4=1/40$\\
 \midrule
 $\frac{h}{\sqrt{2}}=\tau=s$  & 1.60067189E-01 & 3.71949729E-02 & 9.19333755E-03 & 2.28450645E-03 \\
 Rate                         &                & 2.10549806     & 2.01644703     & 2.00870620 \\
 $\frac{h}{\sqrt{2}}=2\tau=s$ & 1.60069253E-01 & 3.71962047E-02 & 9.19375623E-03 & 2.28463221E-03 \\
 Rate                         &                & 2.10546888     & 2.01642910     & 2.00869248\\
 $\frac{h}{\sqrt{2}}=4\tau=s$ & 1.60069873E-01 & 3.71965683E-02 & 9.19387809E-03 & 2.28466840E-03\\
 Rate                         &                & 2.10546037     & 2.01642408     & 2.00868875\\
 $\tau=\frac{2h}{\sqrt{2}}=s$ & 3.71908773E-02 & 9.19191943E-03 & 2.28407438E-03 & 5.69390117E-04\\
 Rate                         &                & 2.01651072     & 2.00875652     & 2.00412027\\
 $\tau=\frac{4h}{\sqrt{2}}=s$ & 9.18720496E-03 & 2.28261124E-03 & 5.68955265E-04 & 1.42028688E-04\\
 Rate                         &                & 2.00894085     & 2.00429804     & 2.00213285\\
 \bottomrule
\end{tabular}}
\label{tab5}
\end{table}
\begin{table}
\caption{Error results for $\max\limits_{n}\|u(t_n)-u_h^n\|_1$ in Example 2 with $\alpha=0.1$.}
\centering
{\begin{tabular}{lllll} \toprule
& $s_1=1/5$ & $s_2=1/10$ & $s_3=1/20$ & $s_4=1/40$\\
 \midrule
 $\frac{h}{\sqrt{2}}=\tau=s$  & 2.47401540E+00 & 1.31270541E+00 & 6.65179344E-01 & 3.33975473E-01 \\
 Rate                         &                & 0.91431129     & 0.98072792     & 0.99400121 \\
 $\frac{h}{\sqrt{2}}=2\tau=s$ & 2.47400989E+00 & 1.31270378E+00 & 6.65179079E-01 & 3.33975433E-01\\
 Rate                         &                & 0.91430986     & 0.98072671     & 0.99400081\\
 $\frac{h}{\sqrt{2}}=4\tau=s$ & 2.47400823E+00 & 1.31270331E+00 & 6.65179002E-01 & 3.33975421E-01\\
 Rate                         &                & 0.91430942     & 0.98072635     & 0.99400069\\
 $\tau=\frac{2h}{\sqrt{2}}=s$ & 1.31271081E+00 & 6.65180241E-01 & 3.33975611E-01 & 1.67229229E-01\\
 Rate                         &                & 0.98073191     & 0.99400256     & 0.99791572\\
 $\tau=\frac{4h}{\sqrt{2}}=s$ & 6.65183228E-01 & 3.33976078E-01 & 1.67229299E-01 & 8.36617901E-02\\
 Rate                         &                & 0.99400702     & 0.99791714     & 0.99918686\\
 \bottomrule
\end{tabular}}
\label{tab6}
\end{table}
%%%%%%%%%%%%%%%%%%%%%%%%%%%%%%%%%%%%%%%%%%%%%%%%%%%%%%%%%%%%%%%%%%%%%%%%%%%%%%%%%%%%%%%%%%%%%%%%%%%%%%%
\begin{table}
\caption{Error results for $\max\limits_{n}\|u(t_n)-u_h^n\|$ in Example 2 with $\alpha=0.5$.}
\centering
{\begin{tabular}{lllll} \toprule
& $s_1=1/5$ & $s_2=1/10$ & $s_3=1/20$ & $s_4=1/40$\\
 \midrule
 $\frac{h}{\sqrt{2}}=\tau=s$  & 1.59933131E-01 & 3.71319625E-02 & 9.17277091E-03 & 2.27822823E-03 \\
 Rate                         &                & 2.10673536     & 2.01723205     & 2.00944532 \\
 $\frac{h}{\sqrt{2}}=2\tau=s$ & 1.59952720E-01 & 3.71461868E-02 & 9.17869712E-03 & 2.28042518E-03\\
 Rate                         &                & 2.10635950     & 2.01685283     & 2.00898655\\
 $\frac{h}{\sqrt{2}}=4\tau=s$ & 1.59959883E-01 & 3.71513345E-02 & 9.18082686E-03 & 2.28121103E-03\\
 Rate                         &                & 2.10622419     & 2.01671804     & 2.00882418\\
 $\tau=\frac{2h}{\sqrt{2}}=s$ & 3.70930838E-02 & 9.15640676E-03 & 2.27212125E-03 & 5.65372956E-04\\
 Rate                         &                & 2.01829676     & 2.01074173     & 2.00676504\\
 $\tau=\frac{4h}{\sqrt{2}}=s$ & 9.11176525E-03 & 2.25530613E-03 & 5.59264054E-04 & 1.38601506E-04\\
 Rate                         &                & 2.01440731     & 2.01172176     & 2.01258668\\
 \bottomrule
\end{tabular}}
\label{tab7}
\end{table}
\begin{table}
\caption{Error results for $\max\limits_{n}\|u(t_n)-u_h^n\|_1$ in Example 2 with $\alpha=0.5$.}
\centering
{\begin{tabular}{lllll} \toprule
& $s_1=1/5$ & $s_2=1/10$ & $s_3=1/20$ & $s_4=1/40$\\
 \midrule
 $\frac{h}{\sqrt{2}}=\tau=s$  & 2.47419989E+00 & 1.31277495E+00 & 6.65190985E-01 & 3.33977313E-01 \\
 Rate                         &                & 0.91434245     & 0.98077909     & 0.99401851 \\
 $\frac{h}{\sqrt{2}}=2\tau=s$ & 2.47414721E+00 & 1.31275608E+00 & 6.65187210E-01 & 3.33976607E-01\\
 Rate                         &                & 0.91433247     & 0.98076654     & 0.99401337\\
 $\frac{h}{\sqrt{2}}=4\tau=s$ & 2.47412796E+00 & 1.31274925E+00 & 6.65185856E-01 & 3.33976355E-01\\
 Rate                         &                & 0.91432874     & 0.98076198     & 0.99401153\\
 $\tau=\frac{2h}{\sqrt{2}}=s$ & 1.31282670E+00 & 6.65201458E-01 & 3.33979287E-01 & 1.67229862E-01\\
 Rate                         &                & 0.98081325     & 0.99403270     & 0.99792614\\
 $\tau=\frac{4h}{\sqrt{2}}=s$ & 6.65230396E-01 & 3.33984826E-01 & 1.67230889E-01 & 8.36620778E-02\\
 Rate                         &                & 0.99407153     & 0.99794120     & 0.99919562\\
 \bottomrule
\end{tabular}}
\label{tab8}
\end{table}
%%%%%%%%%%%%%%%%%%%%%%%%%%%%%%%%%%%%%%%%%%%%%%%%%%%%%%%%%%%%%%%%%%%%%%%%%%%%%%%%%%%%%%%%%%%%%%%%%%%%%%%
\begin{table}
\caption{Error results for $\max\limits_{n}\|u(t_n)-u_h^n\|$ in Example 2 with $\alpha=0.9$.}
\centering
{\begin{tabular}{lllll} \toprule
& $s_1=1/5$ & $s_2=1/10$ & $s_3=1/20$ & $s_4=1/40$\\
 \midrule
 $\frac{h}{\sqrt{2}}=\tau=s$  & 1.59738070E-01 & 3.70003648E-02 & 9.11025687E-03 & 2.24984335E-03 \\
 Rate                         &                & 2.11009679     & 2.02197586     & 2.01766718\\
 $\frac{h}{\sqrt{2}}=2\tau=s$ & 1.59801796E-01 & 3.70596396E-02 & 9.14211637E-03 & 2.26514851E-03\\
 Rate                         &                & 2.10836287     & 2.01924876     & 2.01292254\\
 $\frac{h}{\sqrt{2}}=4\tau=s$ & 1.59831642E-01 & 3.70873701E-02 & 9.15702243E-03 & 2.27231901E-03\\
 Rate                         &                & 2.10755318     & 2.01797751     & 2.01071316\\
 $\tau=\frac{2h}{\sqrt{2}}=s$ & 3.68740143E-02 & 9.04237452E-03 & 2.21733715E-03 & 5.39527731E-04\\
 Rate                         &                & 2.02783091     & 2.02787352     & 2.03905913\\
 $\tau=\frac{4h}{\sqrt{2}}=s$ & 8.89870930E-03 & 2.14908756E-03 & 5.08535496E-04 & 1.15435262E-04\\
 Rate                         &                & 2.04987183     & 2.07930388     & 2.13926449\\
 \bottomrule
\end{tabular}}
\label{tab9}
\end{table}
\begin{table}
\caption{Error results for $\max\limits_{n}\|u(t_n)-u_h^n\|_1$ in Example 2 with $\alpha=0.9$.}
\centering
{\begin{tabular}{lllll} \toprule
& $s_1=1/5$ & $s_2=1/10$ & $s_3=1/20$ & $s_4=1/40$\\
 \midrule
 $\frac{h}{\sqrt{2}}=\tau=s$  & 2.47455235E+00 & 1.31293700E+00 & 6.65229920E-01 & 3.33986479E-01 \\
 Rate                         &                & 0.91436987     & 0.98087273     & 0.99406336\\
 $\frac{h}{\sqrt{2}}=2\tau=s$ & 2.47437911E+00 & 1.31285736E+00 & 6.65209212E-01 & 3.33981391E-01\\
 Rate                         &                & 0.91435638     & 0.98083013     & 0.99404043\\
 $\frac{h}{\sqrt{2}}=4\tau=s$ & 2.47429826E+00 & 1.31282030E+00 & 6.65199619E-01 & 3.33979049E-01\\
 Rate                         &                & 0.91434996     & 0.98081021     & 0.99402974\\
 $\tau=\frac{2h}{\sqrt{2}}=s$ & 1.31310869E+00 & 6.65274972E-01 & 3.33997708E-01 & 1.67234468E-01\\
 Rate                         &                & 0.98096368     & 0.99411256     & 0.99796598\\
 $\tau=\frac{4h}{\sqrt{2}}=s$ & 6.65374613E-01 & 3.34023254E-01 & 1.67241104E-01 & 8.36648725E-02\\
 Rate                         &                & 0.99421828     & 0.99801907     & 0.99923555\\
 \bottomrule
\end{tabular}}
\label{tab10}
\end{table}

%%%%%%%%%%%%%%%%%%%%%%%%%%%%%%%%%%%%%%%%%%%%%%%%%%%%%%%%%%%%%%%%
{\bf Example 2.}\ \
In this example, we consider the equations \eqref{1.1} in two-dimensional space regions,
and choose $\Omega=(0,1)\times(0,1)$, $J=(0,1]$, $q(\bm{x})=1+x_1^2+x_2^2$, $\forall\bm{x}=(x_1,x_2)\in \Omega$,
and
\begin{align*}
   \mathcal{A}(\bm{x})=
   \left(
   \begin{array}{cc}
     2+x_1^2+x_2^2 & x_1^2+x_2^2 \\
     x_1^2+x_2^2 & 2+x_1^2+x_2^2
   \end{array}
   \right),\ \forall \bm{x}=(x_1,x_2)\in \Omega.
\end{align*}
We choose the analytical solution $u(\bm{x},t)=t^2\sin(2\pi x_1)\sin(2\pi x_2)$, then we can get the corresponding initial function $u_0(\bm{x})$ and the source function $f(\bm{x},t)$.

We select some different parameters $\alpha$ and mesh sizes to carry out numerical simulation.
and give some error results with $\alpha=0.1, 0.5, 0.9$ for $u$
in $L^2(\Omega)$-norm and $H^1(\Omega)$-norm in Tables \ref{tab5}-\ref{tab10},
where the mesh sizes are selected as $h=\sqrt{2}\tau, 2\sqrt{2}\tau, 4\sqrt{2}\tau, \frac{\sqrt{2}}{2}\tau, \frac{\sqrt{2}}{4}\tau$.
The error results show that the convergence rates for $u$ in $L^2(\Omega)$-norm are approximate to 2,
and the convergence rates for $u$ in $H^1(\Omega)$-norm are approximate to 1.
Moreover, we also ignore the condition $h\leq c_0\tau$ in Theorem \ref{thm5.2},
fix the time step length $\tau=1E-03$, select the spatial step length $h=\sqrt{2}/10, \sqrt{2}/20, \sqrt{2}/40, \sqrt{2}/80$,
and give error behaviors in Tables \ref{tab11}-\ref{tab12},
in which the convergence rates still satisfy the theoretical results.
From these numerical results in this example,
we can see that the proposed FVE method
for the time fractional reaction-diffusion equations with the Caputo fractional derivative in two-dimensional space regions is feasible and effective.

%%%%%%%%%%%%%%%%%%%%%%%%%%%%%%%%%%%%%%%%%%%%%%%%%%%%%%%%%%%%%%%%%%%%%%%%%%%%%%%%%%%%%%%%%%%%%%%%
\begin{table}
\caption{Error results for $\max\limits_{n}\|u(t_n)-u_h^n\|$ with $\tau=1E-03$ in Example 2.}
\centering
{\begin{tabular}{lllll} \toprule
$\alpha$ & $h_1=\sqrt{2}/10$ & $h_2=\sqrt{2}/20$ & $h_3=\sqrt{2}/40$ & $h_4=\sqrt{2}/80$\\
 \midrule
 0.1  & 1.60070131E-01 & 3.71967163E-02 & 9.19392661E-03 & 2.28468228E-03\\
 Rate &                & 2.10545696     & 2.01642221     & 2.00868760\\
 0.3  & 1.60018261E-01 & 3.71759572E-02 & 9.18810334E-03 & 2.28319506E-03\\
 Rate &                & 2.10579476     & 2.01653090     & 2.00871296\\
 0.5  & 1.59963911E-01 & 3.71541871E-02 & 9.18198205E-03 & 2.28161733E-03\\
 Rate &                & 2.10614975     & 2.01664729     & 2.00874877\\
 0.7  & 1.59909339E-01 & 3.71322390E-02 & 9.17573371E-03 & 2.27992781E-03\\
 Rate &                & 2.10650999     & 2.01677688     & 2.00883538\\
 0.9  & 1.59857457E-01 & 3.71109600E-02 & 9.16926891E-03 & 2.27776720E-03\\
 Rate &                & 2.10686882     & 2.01696671     & 2.00918640\\
 \bottomrule
\end{tabular}}
\label{tab11}
\end{table}
\begin{table}
\caption{Error results for $\max\limits_{n}\|u(t_n)-u_h^n\|_1$ with $\tau=1E-03$ in Example 2.}
\centering
{\begin{tabular}{lllll} \toprule
$\alpha$ & $h_1=\sqrt{2}/10$ & $h_2=\sqrt{2}/20$ & $h_3=\sqrt{2}/40$ & $h_4=\sqrt{2}/80$\\
 \midrule
 0.1  & 2.47400754E+00 & 1.31270311E+00 & 6.65178971E-01 & 3.33975417E-01\\
 Rate &                & 0.91430923     & 0.98072620     & 0.99400065\\
 0.3  & 2.47406090E+00 & 1.31272376E+00 & 6.65181964E-01 & 3.33975808E-01\\
 Rate &                & 0.91431766     & 0.98074240     & 0.99400545\\
 0.5  & 2.47411714E+00 & 1.31274548E+00 & 6.65185122E-01 & 3.33976225E-01\\
 Rate &                & 0.91432659     & 0.98075942     & 0.99401050\\
 0.7  & 2.47417394E+00 & 1.31276744E+00 & 6.65188363E-01 & 3.33976679E-01\\
 Rate &                & 0.91433557     & 0.98077652     & 0.99401557\\
 0.9  & 2.47422847E+00 & 1.31278888E+00 & 6.65191782E-01 & 3.33977288E-01\\
 Rate &                & 0.91434381     & 0.98079267     & 0.99402035\\
 \bottomrule
\end{tabular}}
\label{tab12}
\end{table}

\section{Conclusions}

We apply the FVE methods based on the $L1$-formula
to solve the time fractional reaction-diffusion equations with the Caputo fractional derivative.
We construct the fully discrete FVE scheme,
give the existence and uniqueness analysis,
and derive the stability results which are only depend on the initial data $u_0(\bm{x})$ and the source term function $f(\bm{x},t)$,
where the stability result with $H^1(\Omega)$-norm need to satisfy the condition $h\leq c_0\tau\leq c_0\tau_0$ and $h\leq h_1$.
We also obtain the optimal \textit{a priori} error estimates in $L^2(\Omega)$-norm and $H^1(\Omega)$-norm
by using the properties of the operator $I_h^*$ and some important lemmas.
Moreover, we give two numerical examples, and find that the convergence can still be satisfied in the actual calculation,
even if the condition $h\leq c_0\tau$ is not established.

%\section*{Disclosure statement}
%
%The authors declare that there is no potential conflict of interest.

\section*{Funding}

This work is supported by the National Natural Science Foundation of China (11701299, 11761053, 11661058),
the Natural Science Foundation of Inner Mongolia Autonomous Region (2017MS0107),
the Program for Young Talents of Science and Technology in Universities of Inner Mongolia Autonomous Region (NJYT-17-A07),
and the Prairie Talent Project of Inner Mongolia Autonomous Region.


\begin{thebibliography}{99}\addtolength{\itemsep}{-1.8ex}
%\footnotesize


\bibitem{a1} Hilfer R. Applications of fractional calculus in physics. Singapore: world scientific; 2000.

\bibitem{a2} Magin RL. Fractional calculus in bioengineering. Redding: Begell House; 2006.

\bibitem{a4} Ortigueira. Fractional calculus for scientists and engineers. New York: Springer; 2011.

%%%%%%%%%%%%%%%%%%%%%%%%%%%%
\bibitem{Lin-Xu} Lin Y, Xu C. Finite difference/spectral approximations for the time-fractional diffusion equation. J. Comput. Phys. 2007;225:1533-1552.

\bibitem{Sun-Wu} Sun Z, Wu X. A fully discrete scheme for a diffusion-wave system. Appl. Numer. Math. 2006;56:193-209.


\bibitem{b1} Jiang Y, Ma J. High-order finite element methods for time-fractional partial differential equations. J. Comput. Appl. Math. 2011;235(11):3285-3290.

\bibitem{b2} Liu Y, Du Y, Li H, Li J, He S. A two-grid mixed finite element method for a nonlinear fourth-order reaction-diffusion problem with time-fractional derivative. Comput. Math. Appl. 2015;70:2474-2492.

\bibitem{b21} Liu Y, Du Y, Li H, He S, Gao W. Finite difference/finite element method for a nonlinear time-fractional fourth-order reaction-diffusion problem. Comput. Math. Appl. 2015;70:573-591.

\bibitem{b3} Jin B, Lazarov R, Zhou Z. An analysis of the L1 scheme for the subdiffusion equation with nonsmooth data. IMA J. Numer. Anal. 2016;36(1):197-221.

\bibitem{b4} Zhao Y, Chen P, Bu W, Liu X, Tang Y. Two mixed finite element methods for time-fractional diffusion equations. J. Sci. Comput. 2017;70(1):407-428.

\bibitem{b5} Li D, Liao H, Sun W, Wang J, Zhang J. Analysis of L1-Galerkin FEMs for time-fractional nonlinear parabolic problems. Commun. Comput. Phys. 2018;24(1):86-103.

\bibitem{Li-Huang-Jiang} Li M, Huang C, Jiang F. Galerkin finite element method for higher dimensional multi-term fractional diffusion equation on non-uniform meshes. Appl. Anal. 2017;96(8):1269-1284.

%%%%%%%%%%%%%%%%%%%%%%%%%%%%%%%%%%%%%

\bibitem{Ewing-Lazarov-Lin} Ewing R, Lazarov R, Lin Y. Finite volume element aproximations of nonlocal reactive flows in porous media. Numer. Methods Partial Differential Eq. 2000;16(3):285-311.

\bibitem{c1} Chatzipantelidis P, Lazarov RD, Thom\'{e}e V.  Error estimates for a finite volume element method for parabolic equations in convex polygonal domains. Numer. Methods Partial Differential Eq. 2004;20(5):650-674.

\bibitem{c2} Zhang Z. Error estimates of finite volume element method for the pollution in groundwater flow. Numer. Methods Partial Differential Eq. 2010;25(2):259-274.

\bibitem{c3} Carstensen C, Dond AK, Nataraj N, Pani AK. Three first-order finite volume element methods for Stokes equations under minimal regularity assumptions. SIAM J. Numer. Anal. 2018;56(4):2648-2671.

\bibitem{c4} Zhang T, Li Z. An analysis of finite volume element method for solving the Signorini problem. Appl. Math. Comput. 2015;270:830-841.

\bibitem{c5} Luo Z, Xie Z, Shang Y, Chen J. A reduced finite volume element formulation and numerical simulations based on POD for parabolic problems. J. Comput. Appl. Math. 2011;235(8):2098-2111.


\bibitem{Bank} Bank RE, Rose DJ. Some error estimates for the box methods, SIAM J. Numer. Anal. 1987;24(4):777-787.

\bibitem{Li-Li} Li Y, Li R. Generalized difference methods on arbitrary quadrilateral networks. J. Comput. Math. 1999;17(6):653-672.

\bibitem{Li-Chen-Wu} Li R, Chen Z, Wu W. Generalized Difference Methods for Differential Equations: Numerical Analysis of Finite Volume Methods. New York: Marcel Dekker; 2000.

%%%%%%%%%%%%%%%%%%%%%%%%%%%%%%%%%%

\bibitem{d1} Sayevand K, Arjang F. Finite volume element method and its stability analysis for analyzing the behavior of sub-diffusion problems. Appl. Math. Comput. 2016;290:224-239.

\bibitem{d2} Karaa S, Mustapha K, Pani AK. Finite volume element method for two-dimensional fractional subdiffusion problems. IMA J. Numer. Anal. 2017;37(2):945-964.

\bibitem{d3} Karaa S, Pani AK. Error analysis of a FVEM for fractional order evolution equations with nonsmooth initial data. ESAIM: M2AN, 2018;52:773-801.

\bibitem{Ljch} Li J, Huang Y, Lin Y. Developing finite element methods for Maxwell's equations in a cole-cole dispersive medium. SIAM J. Sci. Comput. 2011;33(6):3153-3174.

\bibitem{Licp1} Li C, Zhao Z, Chen Y. Numerical approximation of nonlinear fractional differential equations with subdiffusion and superdiffusion. Comput. Math. Appl. 2011;62:855-875.

\bibitem{Liufw1} Feng L, Liu F, Turner I, Yang Q, Zhuang P. Unstructured mesh finite difference/finite element method for the 2D time-space Riesz fractional diffusion equation on irregular convex domains. Appl. Math. Model. 2018;59:441-463.

\bibitem{Liufw2} Feng L, Liu F, Turner I. Finite difference/finite element method for a novel 2D multi-term time-fractional mixed sub-diffusion and diffusion-wave equation on convex domains. Commu. Nonlinear Sci. Numer. Simulat. 2019;70:354-371.

\bibitem{Adams} Adams R. Sobolev spaces. New York: Academic Press; 1975.


\end{thebibliography}
\end{document}